\documentclass{amsart}

\usepackage{amsfonts}
\usepackage{amssymb}
\usepackage{amsmath}

\usepackage[arrow,matrix,curve]{xy}
\usepackage{longtable}

\newtheorem{theorem}{Theorem}[section]

\newtheorem{lemma}{Lemma}[section]

\newtheorem{corollary}{Corollary}[section]

\newtheorem{conjecture}{Conjecture}[section]

\newcommand{\cals}[0]{{\mathcal S}}
\newcommand{\calr}[0]{{\mathcal R}}

\newcommand{\call}[0]{{\mathcal L}}

\newcommand{\cale}[0]{{\mathcal E}}
\newcommand{\calf}[0]{{\mathcal F}}
\newcommand{\calk}[0]{{\mathcal K}}

\newcommand{\calg}[0]{{\mathcal G}}
\newcommand{\calz}[0]{{\mathcal Z}}

\newcommand{\N}[0]{{\mathbb N}}

\newcommand{\Z}[0]{{\mathbb Z}}

\newcommand{\C}[0]{{\mathbb C}}

\newcommand{\ad}[0]{{\mbox{ \rm Ad}}}

\newcommand{\ol}[0]{\overline}
\newcommand{\na}[0]{\nabla}
\newcommand{\re}[1]{{(\ref{#1})}}

\newcommand{\be}[1]{\begin{equation} \label{#1}}
\newcommand{\en}[0]{\end{equation}}
\newcommand{\id}[0]{{\rm id}}			

\newcommand{\aut}[0]{{\mbox{ \rm Aut}}}
\newcommand{\triv}[0]{{\mbox{triv}}}

\newcommand{\bfc}[0]{{\bf C}}
\newcommand{\bfd}[0]{{\bf D}} 
\newcommand{\bfj}[0]{{\bf J}}
\newcommand{\bfk}[0]{{\bf K}}
\newcommand{\bfx}[0]{{\bf X}}
\newcommand{\bfv}[0]{{\bf v}}
\newcommand{\bfz}[0]{{\bf z}}
\newcommand{\bfa}[0]{{\bf A}}

\newcommand{\bfy}[0]{{\bf Y}}

\author[B. Burgstaller]{Bernhard Burgstaller}
\title[Computing 
the $K$-homology $K$-theory product in $KK$-theory]{Computing the $K$-homology $K$-theory  product
in splitexact  algebraic $KK$-theory}

\date{\today}

\urladdr{http://mathematik.work/bernhardburgstaller/index.html}

\email{bernhardburgstaller@yahoo.de}

\begin{document}

\begin{abstract} 

Explicit formulas are indicated that 
compute the product $z \cdot w$ of a level-one element $z \in KK^G(A,\bfc)$ 
and any element $w \in KK^G(\bfc,B)$ in splitexact algebraic $KK^G$-theory,  
or $KK^G$-theory for $C^*$-algebras,  
with very special $G$-actions. 
We also make 
such 
products 
accessible to  linear-split half-exact $kk$-theory 
by verifying the existence of a functor 
from algebraic splitexact $KK$-theory to $kk$-theory.

\end{abstract}

\maketitle 

\section{Introduction}

Using the well-known universal property  of the 
$G$-equivariant $KK^G$-theory for $C^*$-algebras introduced by Kasparov \cite{kasparov1981,kasparov1988}, 
namely stability, homotopy invariance and splitexactness \cite{higson}, 
one can immediately define and transfer $KK^G$-theory to
other 
categories of algebras than $C^*$-algebras 
by means of generators and relations implementing 
these axioms 
and calling this completely analogous $KK^G$-theory 
$GK^G$-theory  
for better clarity from now on \cite{gk}. 

In $KK^G$-theory for $C^*$-algebras, every element 
is a so-called level-one element, meaning it has a presentation 
in $GK^G$-theory containing at most one synthetic 
generator split $\Delta_s$ associated to a short splitexact sequence, 
see    
\cite[Theorem 14.3]{aspects}. 
Labeling the collection of level-one elements in $GK^G$-theory 
by $L_1 GK^G$, in this note it is proven 
in Corollary \ref{cor31} and Lemma \ref{cor32} that
the product map in $GK^G$-theory restricts to a map
\be{eq195}
L_1 GK^G (A,\bfc)  \otimes_\Z  GK^G (\bfc, B) 
 \rightarrow L_1 GK^G (A,B) :  z \otimes w \mapsto z \cdot w
 \en
 (the point is we land in $L_1 GK^G$), 
 provided with explicit simple formulas for 
 this product, 
 which are stated in Theorem \ref{thm1},
 if $\bfc= \C$ 
 and $G$ is a discrete group, or more generally if $\bfc = 
 C(X)$ ($X$ compact) 
 and $G$ is a unital discrete inverse semigroup 
 in  
 $GK^G$-theory for algebras, 
 and 
 this 
 method works also 
 in  $KK^G$-theory for $C^*$-algebras.

 However, all 
 this requires 
 the $G$-actions of the involved inverted corner embeddings to be very special,  
 or partially relaxed to be only special
 \big (i.e. involved inverted corner embeddings are of the form $(M_\infty \otimes A, 
 \gamma \otimes \alpha)$ (very special) or their possibly general 
 $G$-action can be extended to the `unitization' $M_\infty(A^+)$
 (special)\big).   
 
We have already found explicit formulas for
 the product map 
 \be{eq196} 
 GK^G (\bfc,A)  \otimes_\Z  GK^G (A,B) 
 \rightarrow L_1 GK^G (\bfc,B) :  z \otimes w \mapsto z \cdot w  
 \en  
 in \cite{gk2},  
 and thus this note 
 advances the tools needed 
 to - potentially - explicitly compute algebraic analogs 
 of 
 $KK^G$-theory products showing up for instance  in the 
 well-known Dirac 
 dual-Dirac method by Kasparov \cite{kasparov1988,zbMATH06685586,arXiv:2210.02332}.     
  
  In $KK$-theory for 
  $C^*$-algebras, the method 
  of calculation of the above product 
   maps  
  \re{eq195} and \re{eq196} 
  for $X$ compact explicitly compute 
  the products 
  $$\calr KK^G \big (X; A, C(X) 
  \big ) \otimes_\Z  
  \calr KK^G \big (X;  C(X),   
  B \big ),$$  
$$\calr KK^G(X; C(X), 
A) \otimes_\Z  
  \calr KK^G(X; A,B),$$
  for `compatible' $KK^G$-theory $\calr KK^G$-theory defined by 
  Kasparov \cite{kasparov1988},  
  and these products are widely used in the literature, 
  given they involve one $K$-theory element as a factor. 
  Provided is however once again that one allows only very special or - partially -  special $G$-actions, which 
  trivially includes 
  all non-equivariant $\calr KK$-theory, but see also Section \ref{sec8} 
  for further comments.   
  Let us give as a short and incomplete sample list where 
  the later products - with arbitrary $G$-actions and not only very special ones - in $KK$-theory play a role: 
  
  Abdolmaleki and Kucerovsky 
  \cite{zbMATH07828316}, 
  Kaad and Proietti 
  \cite {zbMATH07498656},
  van den Dungen 
   \cite{zbMATH07173644}, 
  Debord and Lescure 
   \cite{zbMATH05530175},
  Nishikawa 
   \cite {zbMATH07089428},
  Echterhoff, Emerson  and Kim  
   \cite {zbMATH05248319}, 
  Kaad and Proietti  
   \cite {zbMATH07498656}, 
  Fok and Mathai 
  \cite {zbMATH07366975},
  Baldare   
 \cite{zbMATH07358369}, 
 Androulidakis and Skandalis   
 \cite{zbMATH07155161},
 Arici and Rennie    
\cite{zbMATH07127525},
Mesland and {\c{S}}eng{\"u}n   
\cite{zbMATH07227224}, 
 Zhizhang and Yu    
\cite{zbMATH06317160},
  Dadarlat    
\cite{zbMATH06295993},   
 Kasparov and Skandalis
\cite{zbMATH01997275},
and Kasparov
\cite{kasparov1988,zbMATH06685586, arXiv:2210.02332},
%
%
%
and here we 
did not 
list the many examples 
where a morphism is presented by a single algebra homomorphism,
as for example in 
Antonini, Azzali and Skandalis    
\cite {zbMATH07308574}, 
and 
 Arici and Kaad    
\cite{zbMATH07599758}.  

We would like to remark that we expect the results of these note
and $GK^G$-theory  overall could be extended to 
topological algebras like locally convex algebras, 
but we postpone this to future work. 
 Nevertheless, in anticipation of 
 that generalization, 
  in Section \ref{sec6} we compare Cuntz' $kk$-theory for locally convex 
   algebras \cite{cuntz},
  and its variant for pure algebras by 
  Corti\~nas and Thom \cite{cortinasthom} 
  and 
  from the later one the $G$-equivariant 
  version 
  $kk^G$ by Ellis \cite{ellis}, 
  with $GK$-theory by noting that there 
  exists a functor 
  $$\Gamma:GK^G \rightarrow kk^G$$ 
  which is `identical' 
  on algebra homomorphisms. 
  (This is the analogy to the well-known functor
  $KK \rightarrow E$ from $KK$-theory to $E$-theory - 
  by Higson \cite{higson} and Connes and Higson \cite{conneshigson,conneshigson2} -
  for the category of $C^*$-algebras.) 
   
  The short overview of this note is that in Section \ref{sec2} 
  we recall the 
  definition of $GK^G$-theory, and 
  in Section \ref{sec4} 
  proof the main results stated  above, 
  prepared by a lemma Section \ref{sec3}. 
  In Sections \ref{sec5} and \ref{sec7}  
  we review and summarize which products 
  can be explicitly computed in $GK^G$-theory so far. 
  In Section \ref{sec6} 
  we compare $kk$-theory with 
  $GK^G$-theory. In Section \ref{sec8} we state
  implications of this and other notes about $GK^G$-theory 
   to $KK^G$-theory for $C^*$-algebras.

\section{$GK^G$-theory}			\label{sec2}  

In this section we recall the definition of $GK^G$-theory 
(``Generators and relations $KK$-theory with $G$-equivariance'') 
for algebras, an analogy to $KK^G$-theory for 
$C^*$-algebras \cite{kasparov1981,kasparov1988}. 

We at first remark that this note appears to work also over the category of rings as in \cite{gk} or algebras 
over various fields, but we prefer to select the category of algebras over $\C$ as our base case. 

\subsection{Homotopy}
One may select various notions of {\em homotopy},
for example continuous paths,
smooth paths, 
polynomial
paths, 
but 
rotation homotopies are typically sufficient,
and may be realized by polynomial homotopies \cite[Subsection 3.1]{cortinasthom}. 
Actually, the main results of 
this note do not need homotopy
at all.

\subsection{Algebras and modules}
					\label{sec22}

To be definite, all algebras are regarded as algebras 
over $\C$. 
 Also, all algebras $A$ are supposed to be 
 {\em quadratik}, that is,  
each $a \in A$ 
can be presented as a finite sum of products 
$a_1 b_1 + \ldots + a_n b_n$ in $A$ for $a_i,b_i \in A$. 
(This is a weaker 
condition than having an approximate unit in discrete topology 
on $A$.)  
 
 {\bf (a)} 
 Let $G$ be a {\em discrete group}, or more generally, a {\em unital discrete 
 inverse 
 semigroup} (this essentially covers discrete groupoids $\calg$ with 
 finite base space $\calg^{(0)}$ by regarding $G:=\calg \sqcup \{0\}$ as 
 an inverse semigroup with null element). 
 Write $E \subseteq G$ for the subset of (automatically commuting) {\em idempotent} elements.  
A {\em $G$-action} on an algebra $A$ is a semigroup homomorphism
(unital if $G$ is) 
\be{eq302}
\alpha:G \rightarrow {\rm End}(A) \quad \mbox{such that} \quad  
\alpha_e(a) \cdot b = a  \cdot \alpha_e(b)
\en 
(``{\em compatibility}'')
for all $a,b \in A$ and $e \in E$. 
%
Set  $(\bfc,\chi)$ to be  
the free abelian algebra over $\C$ generated by $E$, 
subject to the 
multiplicative relations that hold in $E$, 
equipped  with the $G$-action 
$\chi_g(e) := g e g^* \in E$ for all $e \in E, g \in G$  \big (in $C^*$-theory one writes:
$\bfc:= C^*(E)$; if $G$ is a group $(\bfc,\chi) =(\C,\triv)$\big). 

{\bf (b)} 
A {\em $G$-action} on a right $(A,\alpha)$-module $\cale$ is a semigroup
homomorphism (unital if $G$ is) $S: G \rightarrow {\rm LinMap}(\cale)$ 
(linear maps) such that 
\be{eq303}
S_g(\xi \cdot a) = S_g(\xi)  \cdot \alpha_g(a) 
\quad  \mbox{and} \quad  S_e(\xi) \cdot a = \xi \cdot \alpha_e(a)  
\en 
(``{\em compatibility}'') 
for all $\xi \in \cale, a \in A, g \in G,
e \in E$. 
The 
$A$-module homomorphism sets 
are equipped 
   with the {\em adjoint $G$-action}, that is, 
   $$\hom_A \big ((\cale,S),(\calf,T) \big )= \big (\hom_A(\cale,\calf), \ad(S,T) \big ),$$ 
   where 
$\ad(S,T)_g (X) :=   
T_g \circ X \circ S_g^{-1}$  
for all $g \in G$.   

{\bf (c)} 
From now on all algebras, modules 
and algebra homomorphisms 
are supposed to be $G$-equivariant 
(but not the module homomorphisms).  
Axiomatically, however, we require all modules homomorphisms 
in $\hom_A(\cale,\calf)$ to be equivariant with respect to the $G$-action restricted 
 to $E$, that is, to be {\em $E$-equivariant}. 
 (Recall this is automatic if $\cale = \cale \cdot A$.) 

 {\bf (d)} 
Algebras $(A,\alpha)$ 
and modules 
are automatically 
{\em $\bfc$-bimodules} with similar formal properties like they are $\C$-modules, induced by the given $G$-action 
on them restricted to $E$,  i.e.  
$$e* a:= a * e := \alpha_e(a) \qquad \forall a \in A,e \in E ,$$
and all $G$-equivariant algebra and module homomorphisms $\varphi$ are 
automatically 
{\em $\bfc$-bimodule maps}
\big (i.e. $\varphi(a * e) = \varphi(a) * e = e* \varphi(a)$\big ).   
%
Consistency of the bimodule 
structure is verified by the 
compatibility axiom of \re{eq302}.  

\subsection{Functional modules}
				\label{sec23}

{\bf (a)} 
 Recall from \cite[Definitions 2.1-2.3]{gk} that 
a {\em functional module} $(\cale,\Theta_A(\cale))$ over a $G$-algebra $A$ is a  
$G$-equivariant 
right $A$-module
$\cale$   together with a $G$-invariant distinguished 
functional space  $\Theta_A(\cale) \subseteq \hom_A(\cale,A)$,  
closed under left multiplication by $A$ 
(one sets $(a \phi)(b):= a\phi(b)$ for $a,b \in A$, $\phi \in \Theta_A(\cale)$). 
An algebra $A$ is regarded as the functional module  
$(A,\Theta_A(A) :=A)$ over itself, where $a (b) := ab$  for $a,b \in A$. 
   Throughout, all modules and algebras are regarded as functional modules. 

{\bf (b)}    
   The {\em adjointable operators} $\call_A(\cale) \subseteq \hom_A(\cale,\cale )$    
is the subset of those 
operators $T \in \hom_A(\cale,\cale )$ satisfying 
$$\phi \circ T \in \Theta_A(\cale)   \qquad 
   \forall \phi \in \Theta_A(\cale), $$ 
see \cite[Definition 2.6]{gk}, and forms a $G$-algebra.     
The {\em compact operators} $\calk_A(\cale) \subseteq \call_A(\cale)$ 
\cite[Definition 2.4]{gk} 
is  the $G$-invariant subalgebra  
(and two-sided ideal)  
of the adjointable operators  
which is the linear span of all  
elementary operators of the form $\Theta_{\xi,\phi}:=\xi \cdot \phi$ for $\xi \in \cale$ and $\phi \in 
\Theta_A(\cale)$.

{\bf (c)}
The module multiplication $\xi * e := S_e(\xi)$ for all $\xi \in (\cale,S)$
and $e \in E$ of Subsection \ref{sec22}.(c) 
is realized by adjointable operators $S_e$ in $\calz \call_{(A,\alpha)}\big ((\cale,S) \big )$ ($\calz$ is center), as 
by \re{eq303}, $S_e \in \hom_A(\cale)$, and 
$\phi \circ S_e 
= \alpha_e \circ \phi \circ S_e = e(\phi) \in \Theta_A(\cale)$, 
and finally $ T \circ S_e = S_e \circ T$ for all $T \in \call_A(\cale)$ 
by $E$-equivariance
 (standing assumption of Subsection \ref{sec22}.(c)). 

{\bf (d)} 
Note that $\call_A(A)$ is a two-sided multiplier algebra for $A \cong \calk_A(A)$ in analogy to $C^*$-theory \big (embedding $A \subseteq \call_A(A)$ by left multiplication operators: $a(b):= ab$ 
for all $a,b \in A$\big ), see also \cite[Lemmas 2.22 and 2.23]{gk}  

For any cardinality $n$, $M_n:= M_n(\C)$ denotes
the algebra of $n \times n$-matrices 
with almost all entries to be zero, and $M_n(A):= M_n \otimes A$.
%
If we want to ignore any $G$-structure or $G$-equivariance 
on algebras or homomorphisms, we sometimes emphasize this by the adjective 
{\em non-equivariant}. 

\if 
 We consider algebras $A$ over $\C$ and right modules $\cale$ over $A$ 
 which are 
   equipped with groupoid actions. 
 That is each 
 A functional module is a $(\cale, \Theta)$
 In this note we use the concept of functional modules
 Normally, in $GK^G$-theory we use the concept of 
 functional modules, adjointable opertors and compact operators. 
  But in this note we strief them only in sideremarks, so that 
  we will not recall them, but instead only 
  define the adjointable operstoe
  
  For an 
  Define $\call_A(A \oplus B)$ 
  
  A functional modules $(\cale,\Theta_A($

 Let $G$ be a discrete groupoid with compact base space 
 $G^{(0)}$, and set $\bfd:= C(G^{(0)})$ (continuous complex valued 
 functions on $G^{(0)}$).  
 In discussions for $C^*$-algebras, $G$ is assumed to be
 a locally compact Hausdorff groupoid with compact base space
 $G^{(0)}$. 
 In analogy to the well-discussed $C^*$-case, 
 all algebras are considered to be $G$-equivariant  
 $\bfd$-algebras, which in partiuclar means 
 that $1_\bfd \cdot  A = A$, and all    
algebra homomoprhism $\varphi:A \rightarrow B$ 
are $\bfd$-compatible, that is, $\varphi (a \cdot d)= \varphi(a) \cdot d$  
We fix a 
set of algebras $R$, and 
\fi

\subsection{Unitization}				

Given 
a 
$G$-invariant subalgebra $X \subseteq 
\call_{(A,\alpha)} \big ((\cale,S)  \big )$ 
one defines 
its {\em unitization}   
$X^+$   
as the smallest $G$-invariant subalgebra 
$X^+$ of $\call_A \big ((\cale  \oplus A, S \oplus \alpha) \big)$ containing $X \oplus 0$ and 
$1:= 1_{\call_A(\cale \oplus A)}$,
which is $(X \oplus 0) + \bfc 1 $, and thus is up to homomorphism well-defined. 
%
(Example: algebra $A \subseteq \call_A(A)$, $A^+ 
\subseteq \call_A(A \oplus A)$.) 

\subsection{$GK^G$-theory}		\label{subsec22}

$GK^G$-theory is 
the {\em universal} 
(possibly polynomial) {\em homotopy invariant}, {\em `matrix-stable'} 
(actually more general)  
and {\em splitexact} theory associated to a set 
of algebras, 
conveniently defined by 
generators and relations \cite{gk}. 

{\bf Generators.} 
Object class is a 
a given {\em set} $R$ of $G$-equivariant algebras $(A,\alpha)$, 
sufficiently big in order to contain all constructions needed,  
and generator morphism class is the {\em set} $\Theta$
comprised 
of 

{\bf (a)} all $G$-equivariant $\bfc$-module algebra homomorphisms between objects of $R$, 
%

{\bf (b)} the synthetic 
morphisms  $e^{-1}: \calk_{(A,\alpha)} \big ((\cale \oplus A, S \oplus A)
\big ) \rightarrow (A,\alpha)$ for 
each  (algebra homomorphism) {\em corner embedding} 
$$e: (A,\alpha) \rightarrow \calk_{(A,\alpha)} \big ( (\cale\oplus A, 
 S \oplus \alpha)
 \big )
 :e(a)(\xi \oplus b)= 0_\cale \oplus   
a b
 $$ 
for all $\xi \in \cale$ and $a,b \in A$, see
\cite[Definition 3.2]{gk}
(the class of functional modules $\cale$ restricted 
to a sufficiently big {\em set} of modules), and  

{\bf (c)} the synthetic 
morphisms $\Delta_{\cals}: X \rightarrow J$, (shortly called 
$\Delta_s$) 
associated to each short splitexact sequence 
of $G$-algebras $J,X,A \in R$ and $G$-equivariant homomorphisms $\iota, f,s$   
(the entered $\Delta_s$ is here not part of the definition) 
\begin{equation}   \label{eq8}
\xymatrix{ 
\cals: 
0   \ar[r]^h     &J     
 	\ar[rr]_{\iota}  &  & X       \ar@<.5ex>[rr]^f    
 	\ar@<-.5ex>[ll]_{ \Delta_{s}}   
& &   
A \ar[ll]^{s}   \ar[r]   
& 0        . 
}
\end{equation} 

{\bf Relations.} 
$GK^G$-theory is then 
the additive category with object class $R$ 
and morphism class 
consisting of 
the 
collection of all valid finite strings 
in the 
alphabet 
$\Theta \cup \{+,-,\cdot,(,)\}$ - valid means
they are meaningful expressions in an additive category 
like $x \cdot y -(a -  e^{-1} \cdot \Delta_s)$  
and the source and range objects fit together -,  
set-theoretically 
divided out by  the generated 
equivalence 
relation which identifies two strings 
if they differ  somewhere by one of  the following elementary equivalences:

{\bf (m)} The associativity and distribution laws, 

{\bf (n)} the abelian and  neutrality laws $x+y=y+x$, $x+0=x$, $x-x=0$ and 
$\id_A \cdot x = x \cdot 
\id_A =x$ 
for all strings $x,y$ (everywhere composability provided), 

{\bf (o)} 
the laws that 
$ f \cdot g = g \circ f$ for all composable algebra homomorphisms $f,g \in \Theta$,   

{\bf (p)} the law that $e^{-1} \in \Theta$ is an inverse element for 
each corner embedding $e \in \Theta$ 
as described in (b), 

{\bf (q)} 
the law that for a homotopy  $f:(A,\alpha) \rightarrow (A,\alpha) \otimes \big (C([0,1]),\triv \big )$ 
\big (i.e. algebra homomorphism, continuous path space $C[0,1]$ may be replaced by
smooth path space $C^\infty([0,1])$ or polynomial path spaces $P(\{0,1\})$\big ) 
one has $f_0 = f_1$ (evaluation maps), 
and 

{\bf (r)} 
the laws that 
$$\id_X = \Delta_s \cdot \iota  + f \cdot s   \quad \mbox{and} 
\quad 
\iota \cdot \Delta_s = \id_J$$    
for each splitexact sequence 
\re{eq8}.  

{\bf Notation.} 
The notation in $GK^G$ (= category $GK^G$-theory) is from left to right, that is, 
 $\xymatrix{  A \ar[r]^{x} & B \ar[r]^y & C}$ composes 
 to $x \cdot y =: xy$, and we solely work with 
 string representatives 
 and do  not use class brackets.

\subsection{Extended double splitexact sequence}    
\label{subsec11}  

An {\em extended double split exact sequence}, see \cite[Definition 9.1]{gk},  
is a diagram 
\be{eq700}
\xymatrix{ 
B  \ar[r]^e      &J   
 	\ar[rr]_{\iota}   & & X        \ar@<.5ex>[rrr]^f    
 	\ar@<-.5ex>[ll]_{ \Delta_{s_-}}   
& & &   
 A \ar[lll]^{s_\pm}    
 }
 \en 
 
in $GK^G$, 
with all objects $A,X,J, B$ being given equivariant 
algebras and 
the  arrows $e, \iota, 
s_+,s_-$ being given equivariant algebra homomorphisms,
such that 

{\bf (a)} $e$ is a corner embedding 
(Subsection \ref{subsec22}.(b)), 

{\bf (b)}  $\iota$ is injective and $\iota(J)$ is an ideal in $X$, 

{\bf (c)}  $X$ is an algebra which is exactly as a set 
the sum $X= \iota(J) + s_-(A)$, and this is a linear direct sum, 
and 

{\bf (d)}  
$s_+,s_-$ are injective and 
\be{middle2}
s_+(a) - s_-(a) \in \iota(J)
\en 
for all $c \in A$. 

Redundantly, the map $f$ in the diagram is always 
forced to be the function, which automatically is an equivariant algebra homomorphism with kernel $\iota(J)$, 
defined by 
$$f \big (\iota(j) + s_-(a) \big ):= a$$  
for all $j \in J$ and $a \in A$.   
Thus, automatically $f$ is surjective and $s_\pm$ are 
injective splits for it.  
Further redundantly, the entered morphism $\Delta_{s_-}$ 
exists then automatically by the splitexactness axiom
of Subsection \ref{subsec22}, (c) and last (r).  
 
 The only deviation from this 
  conventions is that $s_+$ is not necessarily equivariant,
 but actually 
 
 {\bf (e)}  one is by assumption provided with 
 equivariant upper left and lower right corner embeddings 
\begin{equation} 	\label{eq10}
\xymatrix{
	(X,\gamma_-) \ar[rr]^{f_1}  & & (M_2(X),\delta)  &&  \ar[ll]^{f_2}  
		(X,\gamma_+)  } , 
	\end{equation} 
respectively, and 

{\bf (f)}  each $s_\pm : \bfc \rightarrow (X,\gamma_\pm)$ 
are equivariant ($-$ and $+$ separately), 
such that 

{\bf (g)}  $M_2(J)$ is a $G$-invariant subalgebra under $\delta$ 
(this is automatic by \cite[Corollary 6.4]{gk})  
such that the quotient $G$-action on 
$M_2(X)/M_2(J) \cong M_2(X/J)$ 
is of the form $\id_{M_2} 
\otimes \alpha$ for some $G$-action $\alpha: G \rightarrow 
{\rm End}(X/J)$, confer 
\cite[Lemma 6.4]{gk}.  

The {\em morphism associated} to such an extended double splitexact 
sequence is the morphism 
\be{eq29}
s_+ \na_{s_-} :=   s_+ \cdot f_2 \cdot f_1^{-1} \cdot \Delta_{s_-} \cdot e^{-1}  . 
\en

(In other words, one reads the diagram \re{eq700} from right to left, 
that is, 
at first one goes 
along $s_+$, then one changes the $G$-action 
from $\gamma_+$ to $\gamma_-$  
with $f_2 \cdot f_1^{-1}$, then one applies the split $\Delta_{s_-}$,
and finally one `exits' the diagram with the inverse corner 
embedding $h^{-1}$. Note that $\na_{s_-}$ is just an obvious imprecise abbreviation.) 

The $G$-action $\delta$ is called the {\em $M_2$-action} of the 
extended splitexact sequence, 
and $M_2(X)$ its {\em $M_2$-space}. 
Observe, that $\delta$ determines the $G$-actions 
of all algebras of the extended splitexact sequence,  
given all the injective algebra homomorphisms that map 
into the $M_2$-space.    
We use the diagram style and the morphism style interchangeable. 
Elements of the form \re{eq29} are called 
{\em level-one morphisms} or {\em $L_1$-elements}. 
Their collection is coined $L_1 GK^G$ or just $L_1$. 

Diagrams like \re{eq700} are automatically understood to be extended 
double splitexact sequences, even if $\Delta_{s_-}$ is not entered.

\subsection{$KK^G$-theory}		\label{subseckk}

The results of this note may also be interpreted one-to-one in 
$KK^G$-theory, because
if we choose in the definition 
of $GK^G$-theory, Subsection \ref{subsec22},  
for $R$ a suitable set of $C^*$-algebras,
and axiomatically allow {\em corner embeddings} to be the algebra
homomorphisms 
$$e: (A,\alpha) \rightarrow (\calk \otimes A, \gamma): e(a) = e_{11} \otimes a$$ 

 in axioms Subsection \ref{subsec22}, (b) and (p), and $G$ 
 even possibly be a  locally 
compact 
Hausdorff groupoid with {\em compact} base space $G^{(0)}$,  
or an unital inverse semigroup $G$, then 
$GK^G$-theory  is exactly the known $KK^G$-theory, 
that is $GK^G \cong KK^G$, see   
\cite{gk1} or \cite{aspects} for more on this. 
 
 Thus   
 identifying 
 $KK^G = GK^G$ in $C^*$-theory,  
 one has 
 $L_1 GK^G = GK^G$, see  \cite[Theorem 14.3]{aspects}.  
 The quick informal {\em imprecise} translation is that $s_+ \na_{s_-} \in L_1 GK^G(A,B)$ 
 corresponds to something like $[s_- \oplus s_+,\mbox{Flip}] \in KK^G(A,B)$, but see
 \cite[Sections 4 and 9]{aspects} for the precise translation. 
 
 
 \subsection{$M_2$-action} 		\label{sec13}
 
 Typically the
  $M_2$-action $\delta$ is given by restriction as follows:  
  \be{eq77}
 \big (M_2(X),\delta) \subseteq M_2(\call_{A} (\cale) \big ) \cong
 \Big (\call_{(A,\alpha)} \big  ((\cale,S) \oplus (\cale,T) \big ), \ad (S \oplus T) \Big ),
 \en 

 where $S$ and $T$ are 
 $G$-module 
 actions on the right $(A,\alpha)$-module $\cale$, 
 and $X$ is a subalgebra of $\call_A(\cale)$.  
 Note that the formula for the $G$-action is 
 \be{eq300}
 \ad (S \oplus T)_g 
 \left ( \begin{matrix}  a  &  b  \\  
 c  &  d
  \end{matrix}  \right  )    =  
  \left  ( \begin{matrix}  S_g  a  S_{g^{-1}}  &   S_g  b   T_{g^{-1}}  \\  
 T_g  c S_{g^{-1}} &  T_g  d   T_{g^{-1}}
  \end{matrix}  \right  )   
  \en
  for all $a,b,c,d \in \call_A(\cale)$ and $g \in G$ (here, $S_g a S_{g^{-1}} =  S_g  \circ a  \circ S_{g^{-1}}$ et cetera). 

Actually, every $M_2$-action is of the above form, 
see 
\cite[Lemma 6.3]{gk}. 

\subsection{Commuting diagrams}			\label{subsec24}

Typically, identities in 
$GK^G$-theory are proven by 
commuting diagrams between two $L_1$-morphims:  

Consider the first two lines of the diagram of Theorem 
\ref{thm1} below and assume that 
they are any extended double splitexact sequences $s_+ \na_{s_-}$ 
and $u_+ \na_{u_-}$, respectively, 
as explained above. 
The second line is allowed to 
consist of {\em any} other algebras and 
algebra homomorphisms and need not be the entered ones. 
Assume that the diagrams are connected by {\em any} morphisms
$k,l,m,e$ as entered in the diagram, where $m$ is an algebra homomorphism. 
Then if this picture commutes in 
$GK^G$-theory, that is 
$$h l = k H, \quad \iota m = l \iota_n, \quad   
s_+ m = e u_+, \quad \mbox{and}   \quad 
s_- m = e u_-$$ 
(the $f$s and $\Delta$s  can be ignored), 
and the map $M:=m \otimes \id_{M_2}$ is an {\em 
equivariant} algebra homomorphism between the two 
$M_2$-spaces of the two lines of the diagram, 
then one has 
$$ s_+ \na_{s_-} \cdot k =  e \cdot u_+ \na_{u_-} . $$

The proof follows directly from the similar \cite[Lemma 4.3]{gk}, 
because the requirement 
$f s_- m = m (f \otimes \id_{M_n}) u_-$ 
is easily checked by writing elements of $X$ as 
$x= \iota(j) + s_-(c)$ and using the formula for the $f$s 
from Subsection \ref{subsec11}. 
 
\subsection{Middle space} 			\label{subsec28}


Let $s: (A,\alpha) \rightarrow (M,\gamma)$ be a given algebra homomorphism 
and $(J, \gamma|_J)  \subseteq 
(M,\gamma)$  be a 
$G$-invariant, two-sided ideal in $M$. 
Then one constructs the 
$G$-invariant, $G$-equivariant  {\em subalgebra} 
of  $(M \oplus A, \gamma \oplus \alpha)$ given by  
%
\begin{equation}   \label{middle}
M \square_s A := \{(j \oplus 0) + (s(a) \oplus a) 
 \in M \oplus A
  | \; j \in J, a \in A\}   . 
  \end{equation} 
  
  (The ideal $J$ is imprecisely not notated. 
  We may also notate it as $M \square_{s \oplus \id_A} A$ 
  if it is more convenient.)  
 Linearly this is a direct sum $J \oplus A$, and thus we define 
 the {\em algebra} $(J \stackrel{s}{\oplus} A, \gamma|_J \oplus \alpha)$ which is the linear space $J \oplus A$ equipped with the multiplication 
 $$(j_1 \oplus a_1) \cdot (j_2 \oplus a_2) 
 :=  j_1 j_2 + j_1 s(a_2) + s(a_1) j_2 \oplus a_1 a_2  . $$

  
  Thus we get an equivariant algebra isomorphism
  \begin{equation}   \label{zeta}   
  \zeta:  (J \stackrel{s}{\oplus} A, \gamma \oplus \alpha) 
  \rightarrow (M \square_s A , \gamma \oplus \alpha) : 
   \zeta (j \oplus a) := j + s(a) \oplus a   .  
  \end{equation}    

 \subsection{Special $G$-actions}     \label{subsec26} 
 
 A $G$-action  $\delta$ on $M_n(A)$ is called {\em special} 
 if it can be extended to a $G$-action $\delta_2$ on $M_n(A^+)$
 \big (the `{\em unitization}' of $M_n(A)$\big ). 
  It is called {\em very special} if it is even of the form $\delta
  = 
  \gamma \otimes \alpha$ for $M_n= (M_n,\gamma)$ and $A=(A,\alpha)$. 
    
\big (Consequently,  a canonical matrix corner embedding $e: (A,\alpha) \rightarrow (M_n(A),\delta)$ and an $M_2$-space $(M_2(A),\delta)$, respectively,  
 are 
 called special 
 if $\delta$ is, et cetera.\big )
 
 $GK^G$-theory is said to be {\em 
 very special} 
 if only 
 very special, canonical matrix corner embeddings are 
 axiomatically declared to be invertible 
 in axioms Subsection \ref{subsec22}, (b) and (p).  
 %

 If the $M_2$-space of an $L_1$-element 
 is very special, then it can only 
 have the $G$-action of the form $\id_{M_2} \otimes \gamma$, 
 and thus one can omit the corner embeddings $f_1,f_2$ in \re{eq10} 
 and \re{eq29} at all, because $f_1 = f_2$ in $GK^G$-theory 
 by an ordinary rotation homotopy.

 \subsection{Improved matrix corner embeddings}
				\label{subsec211} 

 Up to now,   we can show that $\bfz \cdot e^{-1} \in L_1$ 
 by \cite[Corollary 8.2]{gk} 
 but not $e^{-1} \cdot \bfz  \in L_1$
 for $\bfz \in L_1$  
 and {\em general} corner embeddings $e : A \rightarrow \calk_A(\cale \oplus A)$ (cf. Subsection \ref{subsec7}),  
 but by our recent preprint \cite{gk2} 
 for many important modules appearing 
 in 
 practice, namely those $\cale$ of the form 
 \begin{longtable}[ht]{ccc}   
   $(A \otimes_\pi B, \alpha \otimes \beta)$ &  
   $( A \otimes B , \alpha \otimes \beta)$   & 
   $(\oplus_i A_i, \oplus_i \alpha_i)$ \\ 
   change of coefficient algebra & external tensor product 
  & direct sum  \\
 \end{longtable}
 and  mixed combinations 
 thereof,  $e$ is invertible 
 and 
 both products $\bfz \cdot e^{-1} \in L_1$
 and $e^{-1} \cdot \bfz \in L_1$ are computable within 
 very special $GK^G$-theory 
 for the category of discrete algebras having approximate units 
 and when $G$ is a group.  
 
 In these cases 
 $$e^{-1} = \phi \cdot f^{-1}$$
 for an injective algebra 
 homomorphism $\phi$  and $f: (A,\alpha) \rightarrow 
 ( M_n \otimes A, \gamma  
 \otimes \alpha)$ a canonical, very special corner embedding 
 by \cite[Proposition 2.11, Proof item 
 (c)]{corner} 
 This format will be taken into account in 
 Theorem \ref{thm1} below.

 For $G$ a compact group we have the following 
 result from \cite{corner}, potentially particularly useful for $KK^G$-theory:  
 
 \begin{lemma}[Compact $G$ and very special actions]
 \label{lemma21} 
  If $G$ is a finite group, then $GK^G$-theory 
  for the category of discrete algebras having approximate units,   
  and where only 
  canonical matrix corner embeddings $A \rightarrow \calk_{(A,\alpha)}((A,\alpha) \oplus (A^n, \gamma))$ (any $G$-action $\gamma$, $n$
  any permitted cardinality) are 
 axiomatically declared to be invertible 
 in axioms Subsection \ref{subsec22}, (b) and (p),   
%
  is 
  very special. 
The axiomatically declared invertible corner embeddings $e$ in $GK^G$
fulfill the same format 
$e^{-1} = \phi \cdot f^{-1}$ as described before.  
\end{lemma}

\begin{proof}

That one gets very special $GK^G$-theory follows from the fact 
that  
one can write $e^{-1} = \phi \cdot f^{-1}$ for a very special corner 
emebdding $f$, which shows 
that $e$ is already invertible in very special $GK^G$-theory 
and its invertibility axiom thus redundant there. 
That this format is met for $e^{-1}$ is proven in
\cite[Corollary 5.3]{corner}, proof step (b), refering to 
\cite[Proposition 2.11]{corner}, proof step (c), where finally the 
formula appears. 
\end{proof}

\section{Some lemmas} 			\label{sec3} 
 
 In this section we collect some lemmas needed 
 in the next 
 section of main results. 

  The next lemma says that the product  
 of an inverse {\em special} corner embedding 
 with a level-one element 
 is a level-one element again:  

 \begin{lemma}			\label{lemma31} 
 
 Any level-one element $s_+ \na_{s_-} \in L_1 GK^G$ 
 with special $M_2$-space 
 $$\big ( M_2 \big (X),  \ad(S\oplus T) \big)$$  
 can be fused with the inverse of a special 
 corner embedding $e$ to a level-one element
 $u_+ \na_{u_-}=  e^{-1} \cdot s_+ \na_{s_-} \in L_1 GK^G$ 
 (provided composability), see
 the first two lines of the diagram of 
 Theorem \ref{thm1} ($u_\pm$ have the indicated formulas).  
 
 If the matrix size $n$ is finite,    
 then		
 $u_+ \na_{u_-}$ has  special $M_2$-space. 

 If $e:(\bfd, \delta) \rightarrow (\bfd ,\delta) \otimes (M_n, 
  \sigma)$ is even very special,
  then $s_+ \na_{s_-}$ may have any (possibly non-special) $M_2$-action, 
  and 
    the $M_2$-space of  $u_+ \na_{u_-}$ 
 is 
 $$\big ( M_2(X \otimes M_n), \ad(S \otimes \sigma  \oplus 
 T \otimes \sigma) \big ), $$   
 
 which is very special if the $M_2$-space of $s_+ \na_{s_-}$  
 is very special.    

 \end{lemma}
 
 \begin{proof}
 
 Superficially we immediately get a commuting diagram 
between the first and second line  
of the diagram of Theorem \ref{thm1} 
- without the
$M_2$-action - but the problem is to find an $M_2$-action for the second line, and this is clarified in 
the proof step ``Line 2'' of 
\cite[Proposition 3.4]{gk2} employing 
\cite[Propsition 9.7]{gk}, which yields 
the claim. (The method of proof is to 
identify  
$\bfd^+  \otimes_{{s_\pm}^+} X^n \cong X^n$
non-equivariantly, where we see why we need speciality of 
all involved matrices, namely their `unitizations'.)

The case $n < \infty$ is also discussed there 
(in proof step ``Line 2'', item (f) of 
\cite[Proposition 3.4]{gk2}) 
and shows that the $M_2$-space of $u_+ \na_{u_-}$ is 
then 
special.  
 For the last assertion, when $e$ is very special, 
 we take the proposed $M_2$-action and verify 
 the claim with commuting diagrams (Subsection \ref{subsec24}).
%
 %
 %
 If the $M_2$-space of $s_+ \na_{s_-}$ is very special then
 necessarily $S=T$, yielding the last claim. 
\end{proof}



This gives an equivalent condition when an $M_2$-space is special: 

\begin{lemma}			\label{lemma32} 

Let $(M_2(X), \delta)$ be a $G$-algebra which leaves 
the upper 
left and lower right corner algebra $G$-invariant (which implies that $\delta=(\delta^{ij})_{1 \le 
i,j \le 2}$ (matrix-form)). 
%
Assume that $Z:=  J + A$ 
is a non-equivariant subalgebra of the supposed unital algebra $X$ which is the linear direct sum 
of a two-sided ideal $J$ 
and an subalgebra $A$ 
of $Z$.   

Then $\delta$ restricts to 
a special $G$-action on 
$M_2(Z)$ 
under which $M_2(J)$ is invariant and 
such that the quotient action on $M_2(Z^+)/M_2(J) \cong M_2(Z^+/J)$ 
is of the form $\id_{M_2} 
\otimes \alpha$,  
if and only if  
\be{eq111}
\mbox{$J$ and 
	$Z$ are invariant under $\delta^{11}$}  
\en
$$\delta^{21}_g (1_X) - \delta^{11}_g (1_X) \in J, \; 
\delta^{12}_g (1_X) - \delta^{11}_g (1_X) 
\; \in J 
\qquad (\forall g \in G) . $$  
\end{lemma}

\begin{proof}

{\bf (a)} 
The form of $\delta$ is by \cite[Lemma 6.3]{gk}. 
We may assume that $1_X \notin Z$, because if not, 
we may extend $X$ to $X^+$ and 
extend $\delta$ obviously
to $\delta \oplus \chi$ acting on $M_2(X^+) \cong M_2(X) \oplus M_2(\bfc)$ 
(as $X$ is unital) and
use for it the notation $\delta$ again; 
note that also condition \re{eq111} would be unaffected under such a change.

{\bf (b)}  
Since the upper left corner action $\delta^{11}$ is a $G$-algebra action, we have $\delta^{11}_g (1_X) \delta^{11}_g(a) = \delta^{11}_g(a)$. 
%
%
Consequently, 
\be{eq112} 
\delta^{21}_g (a) =  \delta^{21}_g (1 )  \cdot \delta^{11}_g(a)  
= \big (\delta^{21}_g (1 ) -  \delta^{11}_g (1 ) \big ) \cdot \delta^{11}_g(a) 
		+ \delta^{11}_g(a)    
		\en 
\be{eq113} 
\delta^{12}_g (a) =  \delta^{11}_g (a ) \cdot  \delta^{12}_g(1)  
= \delta^{11}_g(a)  \cdot \big ( \delta^{12}_g (1 ) - \delta^{11}_g (1 ) \big )  
		+ \delta^{11}_g (a )    .  
\en

{\bf (c)} 
If condition \re{eq111} holds true, then 
 identities \re{eq112}  and \re{eq113} 
together with $\delta^{22}_g (a) =  \delta^{21}_g (a )  \delta^{12}_g(1)$ show that 
$Z + \C 1_X = Z^+$ is invariant under all $\delta^{ij}$ .
They also show that 
$\alpha(a + J) := \delta^{ij} (a) + J = \delta^{11}(a) + J $ 
for all $a \in Z + \C 1_X$ and $i,j$, and thus the 
claimed format $\id_{M_2} \otimes  \alpha$.

{\bf (d)} 
If on the other hand the other condition holds, 
then because of the supposed format of the quotient action 
on $M_2 (Z^+/J)$ one
has $^\alpha (1_x +J)=\delta^{ij} (1_X) + J = \delta^{11}(1_X) + J = 1_X + J$ for all $i,j$.
 Putting $a:=1_X$ in the above identities \re{eq112}  and \re{eq113},
  we see that we must 
 get   condition \re{eq111}.  
 %
\end{proof}

The following lemma 
used later in the proof of the main theorem holds also in $C^*$-theory. 


\begin{lemma}		\label{lemma22}

Let $A$ be 
a $G$-algebra 
which has an increasing 
sequence $(p_n)$ of 
projections (i.e. idempotents in the algebraic case) which also forms
an approximate unit of $A$.  

{\rm (i)}  
If 
a $G$-equivariant algebra  homomorphism $\varphi  : A \rightarrow \call_X(X)$  
satisfies $\lim_n \varphi(p_n) (x) = x$ for all $x \in X$,     
then
it can be extended to 
the  $G$-equivariant algebra homomorphism 
$$ \overline \varphi : \call_{A}( A) 
\rightarrow \call_{X}(  X) : \overline \varphi(U) 
(x) := \lim_{n \rightarrow \infty} \varphi (U \cdot p_n) (x) 
\qquad (\forall U \in \call_A(A), x \in X)   . 
$$   


{\rm  (ii)}  
If another  homomorphism $s: A \rightarrow \call_X(X)$ satisfies 
$$s(a \cdot b ) = s(a) \cdot \varphi(b) = \varphi(a) \cdot s(b)$$ 
for all
$a,b \in A$  (i.e. $s$ is a bimodule map), then 
also $s(a \cdot V ) =  s(a) \cdot \overline \varphi(V)$
and $s( V \cdot a ) 
= \overline \varphi(V)  \cdot   s(a)$ 
for all $a \in A,V \in \call_A(A)$.  

{\rm (iii)}  
If $s$ satisfies also the assumption of (i) then also
 $\overline s(U \cdot V ) = \overline s(U) \cdot \overline \varphi(V) 
= \overline \varphi(U)  \cdot  \overline s(V)$ 
for all $U,V \in \call_A(A)$.  

\end{lemma} 

\begin{proof}

{\bf (a)} 
(i) 
We prove it for $C^*$-theory, where
$\call_A(A)$ denotes the well-known algebra of adjointable 
operators on the Hilbert modules $A$ over itself. 
Given $U \in \call_A(A)$ and $x \in X$,
choose $n_0 \ge 1$ such that for all $n \ge n_0$ one has 
$\|\varphi(p_n) x - x\| \le \varepsilon$ (by assumption). 
Then for all $N \ge n \ge n_0$ one has $p_N \ge p_n$ and thus 

$$\| \varphi(U \cdot p_n) x   -\varphi(U \cdot p_N) x \|
=  \| \varphi(U \cdot p_n) x  -\varphi(U \cdot p_N)
\cdot \varphi(p_N)  x  \|$$
$$\le \| \varphi(U \cdot p_n) (x)  -\varphi(U \cdot p_N)
\cdot  \varphi(p_n)  x \| + 2 \varepsilon \|U\| 
= 2 \varepsilon \|U\|, $$
verifying convergence of the formula of $\ol \varphi$. 
Then $\ol \varphi$ is a homomorphism: 
$$\lim_n \varphi (U \cdot  V \cdot p_n)x  
= \lim_{n} \lim_m \varphi \big (U \cdot p_m \cdot (V \cdot p_n) \big ) x 
= \lim_{n,m} \varphi (U \cdot p_m) \cdot \varphi( V \cdot p_n ) x . $$

We see 
that 
$\varphi(U)$ is an adjointable operator in the sense of $C^*$-theory: 
$$\langle \ol \varphi(U) x ,  y \rangle  = 
\langle \lim_{n \rightarrow \infty} \varphi (U \cdot p_n) x, 
y \rangle =  
 \lim_{n,m \rightarrow \infty}  \langle x, 
\varphi ( p_n^* \cdot U^* \cdot p_m) \rangle 
= \langle x , \ol \varphi( U^*) y \rangle  . 
$$

{\bf (b)} 
(ii) \& (iii) 
Given $a \in A, U,V \in \call_A(A)$ and $x \in X$, we verify 
(ii) and (iii) with
%
%
$$s ( a \cdot V)   x  
= \lim_n s ( (a \cdot V) \cdot p_n)  x  
= \lim_n s ( a ) \cdot \varphi(  V \cdot p_n) x 
= s ( a ) \cdot \overline \varphi( V) x  ,$$
$$\overline s ( U \cdot V)   x  = \lim_n  s ( U \cdot V \cdot p_n)  x  
= 
\lim_n \lim_m 
s ( U \cdot p_m \cdot (V \cdot p_n) ) x  
=   \overline s ( U ) \cdot \overline  \varphi( V) x  . $$

Write $g(U):= \alpha_g \circ U \circ \alpha_g^{-1}$ 
for the $G$-action on $\call_A(A)$ ($g \in G, U \in \call_A(A)$),
also analogously on $\call_X(X)$. 
 
Then, by $G$-equivariance of $\varphi$, 
we obtain $G$-equivariance of $\ol \varphi$:  
$$\overline \varphi \big ( g(U) \big ) x  = \lim_{n \rightarrow \infty} 
\varphi \big ( g(U) \cdot p_n  \big ) x
= \lim_{n \rightarrow \infty} 
g \big (\varphi ( U \cdot g^{-1}(p_n) ) \big ) x
$$
$$= \lim_{n \rightarrow \infty} 
g \big (\ol \varphi ( U) \big )  \circ  g \big ( \varphi(g^{-1}(p_n) ) \big ) x 
= \lim_{n \rightarrow \infty} 
g \big ( \ol \varphi ( U) \big )  \circ \varphi(p_n)  x 
= g  \big (\ol \varphi ( U)  \big )    x   . 
$$

\end{proof}

   \section{Main results}			\label{sec4} 

 In this section we compute the $K$-homology $K$-theory 
 product in $GK^G$-theory and $KK^G$-theory in
 Corollary \ref{cor31} below. 

 The following theorem roughly 
 says that in
 very special $GK^G$-theory the product 
 $\xymatrix{  B \ar[r]^{x} & \C \ar[r]^y & A}$ of level-one 
 elements $x,y \in L_1$ 
 is a level-one element 
 again.

\begin{theorem}			\label{thm1}

Let $s_+ \na_{s_-} \in L_1 GK^G( 
\bfd, A)$
and  $t_+ \na_{t_-} \in L_1 GK^G( B, \bfd)$  
be any 
given level-one elements as indicated in the first line
and last row, respectively,  of the following diagram. 

$$\xymatrix{ 
A  \ar[r]^h   \ar[d]^k   &J   
\ar[d]^l      
 	\ar[r]_{\iota}   & X      \ar[d]_m  \ar@<.5ex>[rrr]^f    
 	\ar@<-.5ex>[l]_{ \Delta_{s_-}}   
& & &   
 \bfd \ar[lll]^{s_\pm}    \ar[d]_e   \\
A  \ar[r]^{  H }  
   &    M_n (J) 
\ar[r]^{ \iota_n:= 
\iota \otimes \id_{M_n}}          &    
M_n (X)   
\ar@<.5ex>[rrr]^{ 
f  \otimes \id_{M_n } }   
& & &    M_n (  \bfd)   
\ar[lll]^{
u_\pm := s_\pm \otimes \id_{M_n  }}     
\ar@<-.5ex>[u]_{ e^{-1}}   
\\  
A  \ar[r]^H   \ar[d]  \ar[u]  &    \bfj  := M_n(J)  
\ar[r]^-{
  \jmath :=\iota_n \oplus 0} 
\ar[d]   
\ar[u]   &    
  \bfx  := M_n(X) 
   \square_{S_-} \bfk   \ar[u]^{  \id  
   \oplus 0}
\ar@<.5ex>[rrr]^{F := 0 \oplus \id_{\bfk}}   
\ar[d]_{\id_{\bfx}  \oplus 0_B }    & & &    \bfk     \ar[u]^\phi 
\ar[lll]^{S_\pm:= u_\pm  
 \circ \phi \oplus \id_{\bfk}  }     \ar[d]_{\id_{\bfk}  
\oplus 0_B }      
     \\  
A  \ar[r] 
\ar[d]  &   
\bfj  \ar[r]^{ \jmath 
	\oplus 0_B} 
\ar[d]  &    
\bfx
\stackrel{T_-}{\oplus}       B    
\ar@<.5ex>[rrr]^{   
F \oplus \id_B }      \ar[d]^{E_1}  
  & &  &  
  \bfk 
\stackrel{t_-}{\oplus}   
  B   
  \ar[lll]^{   
S_\pm \oplus \id_B}     \ar[d]^{e_1}  
\ar@<-.5ex>[u]_{ \Delta_{t_-}} \\
A  \ar[r] 
&   
M_2(\bfj)   \ar[r]^{( \jmath 
	\oplus 0_B) \otimes \id_{M_2} } 
&    
M_2 \big ( \bfx
\stackrel{T_-}{\oplus}       B    
\ar@<.5ex>[rrr]^{   
(F \oplus \id_B)   \otimes \id_{M_2} } \big )      
  & &  &  
  M_2 \big ( \bfk 
\stackrel{t_-}{\oplus}   
  B    \big ) 
  \ar[lll]^{   
v_\pm:=  (S_\pm \oplus \id_B) \otimes \id_{M_2}}     \\  
A  \ar[r]   \ar[u] 
\ar[d]    &   
\bfj  \ar[r]^{\jmath  
	\oplus 0} 
\ar[d]   \ar[u]   &    
\bfx
\stackrel{T_-}{\oplus}       B    
\ar@<.5ex>[rrr]^{   
F \oplus \id_B }      \ar[d]^{\zeta_2}   \ar[u]^{E_2}     
  & &  &  
  \bfk 
\stackrel{t_-}{\oplus}   
  B   
  \ar[lll]^{   
S_\pm \oplus \id_B}     \ar[d]^{\zeta_1}    
\ar[u]^{e_2}   \\
A  \ar[r]    &   
\bfj   \ar[r]^{\jmath   
	\oplus 0}   &    
\call \bfx 
 \square_{T_-}        B    
\ar@<.5ex>[rrr]^{   
   F(x - T_-(b))  + t_-(b)   \oplus b  }   
  & &    &
   \call \bfk 
 \square_{t_-}   
  B   
  \ar[lll]^{   
S_\pm( k 
- t_-(b)) + T_-(b) \oplus b}  
\ar@<.5ex>[d]^{0 \oplus \id_B}  
 \\
A  \ar[r]  \ar[u]    &   
 \bfj  \ar[r]   \ar[u]    &   
  \call  \bfx   
\square_{x_-} 
  B  
\ar@<.5ex>[rrr]^{0 \oplus \id_B}   
\ar[u]^{ \id }   
  & &    &
   B   
	\ar[u]^{ t_\pm  \oplus \id_B}   
	 \ar[lll]^{  x_\pm  = 
 \big ( S_\pm  \circ (t_+ - t_-)  + T_-  \big )  \oplus \id_B}  
}$$ 

Let the $M_2$-spaces of both $s_+ \na_{s_-}$ (the first line) and 
$t_+ \na_{t_-}$ (the last column) be special, and the corner embedding $e$ 
in the last column of the above diagram be very special. 
  If 
  the $M_2$-space of $s_+ \na_{s_-}$  
  is very special, the matrix size $n$ can be infinite, 
  otherwise it must be finite.   %
  %

In $t_+ \na_{t_-}$, 
the morphism $\phi \cdot e^{-1}$ in the last column is regarded 
as the inverse of a corner embedding, where $\phi$ may be any 
algebra homomorphism. 
The maps $\zeta_1$ and $\zeta_2$ 
are just the algebra isomorphisms of Subsection \ref{subsec28}. 


If
assuming that - given the $M_2$-space of the first line of the above diagram  is denoted by $ \big (M_2(X), \ad(\gamma_- \oplus \gamma_+) \big)$, 
see \re{eq77} - 
there exists an algebra homomorphism 
$$\varphi: (\bfd, \delta 
) \rightarrow \big (\call_X(X), \ad(\gamma_\pm) \big )$$ 
which is 
equivariant with respect to both $\ad(\gamma_-)$ and 
$\ad(\gamma_+)$
- \big (this is the case if it is equivariant with respect to one of them
and $\varphi$ maps into the center of $\call_X(X)$ \big ) -
such that
$$s_\pm(c \cdot d) = s_\pm(c) \cdot \varphi(d) = \varphi(c) 
\cdot s_\pm(d)$$ 

for all $c,d \in \bfd$ 
(i.e. $s_\pm$ are 
$\bfd$-bimodule maps),  
and $\bfk$ has an approximate unit which is an increasing 
net $(p_m)$ 
of projections, 
then the product of the given level-one elements 
is a level-one element $x_+ \na_{x_-} \in L_1 GK^G(B,A)$ in $GK^G$ as indicated in the last line of the diagram, that is, 
$$t_+ \na_{t_-} \cdot s_+ \na_{s_-} = x_+ \na_{x_-}$$
or written out 
$$
x_\pm (b) = 
 \Big (   (s_\pm \otimes \id_{M_n  }  )  
 \circ \phi \oplus \id_{\bfk} \Big )    \circ (t_+ - t_-) (b)  \oplus b
 $$
 $$  {+ 
   	{
   \lim_{m \rightarrow \infty} \Big ( 	\tau \circ (\varphi \otimes \id_{M_n})  \circ \phi \oplus  \id_{\bfk} \Big )} \big( t_-(b) \cdot p_m \big )  
    \oplus b }  . $$
   
   	\big (Here, $t_\pm: B \rightarrow \call_\bfk(\bfk)$, 
 	the ideal of $t_+ \na_{t_-}$ is $\bfk$ notated in the third line of the above diagram, such that $t_+ -t_-$ 
 	maps to $\bfk$, and $\tau$ is from \re{eq67}.\big ) 

More precisely, here we set
\begin{equation}    \label{eqt}
t_+ \na_{t_-}:= (t_+ \oplus \id_B) \cdot  
\zeta_1^{-1} \cdot  e_2 \cdot e_1^{-1}  
 \cdot \Delta_{ \zeta_1^{-1} \circ (t_- \oplus \id_B)}  \cdot \phi \cdot e^{-1}    .  
 \end{equation} 

This theorem does not need the homotopy axiom of $GK^G$-theory. 

Also, this theorem 
holds 
analogously exactly in that form in $KK^G$-theory for $C^*$-theory 
by interpreting $KK^G \cong GK^G$ \cite{aspects}  
and replacing $M_n$ by $\calk$ (compact operators on Hilbert space) 
if $n= \infty$, 
and only with special and very special
$G$-actions as indicated. 


\end{theorem}


\begin{proof}

{\bf Strategy.} 
We are given the first line (= $s_+ \na_{s_-}$) and the last column 
(= $t_+ \na_{t_-}$)
of the above diagram, and our strategy of proof 
is to fuse step by step at first the first line with the right most 
vertical arrow glued to it (= $e^{-1}$) to obtain the second line by verifying 
a commuting diagram between lines one and two,  
and so on, until we end up at the last line
(= $x_+ \na_{x_-}$). 
All steps are known techniques to fuse algebra homomorphisms 
or inverse corner embeddings with $L_1$-elements,
the exception being the step from line three to four, which 
is the novelty.

{\bf Line 1: } 
At first we note that as remarked in Subsection \ref{subsec11} 
we are automatically provided with the maps 
$f$s in the above diagram and we have only entered them 
informatively.
The discussion and verification of the involved $M_2$-spaces 
is postponed to the end of the proof, and thus
 we  {\em at first  
ignore all indicated $G$-actions} until then. 
Also throughout the proof, the above diagram of the 
theorem is referred to as the ``above diagram''.  

Also regard and interpret the last column $t_+ \na_{t_-}$ of the above diagram as an $L_1$-element 
by assuming that  $\bfk:= M_n(X)$ and $\phi= \id_{M_n(X)}$ 
for a second: \re{eqt} is then $\big (\zeta_1^{-1} \circ ({t_+ \oplus \id_B}) \big )   
\na_{\zeta_1^{-1} \circ (t_- \oplus \id_B)}$ and meets 
the format \re{eq29}, where 
the $e_1,e_2$-arrows are the analogs
of \re{eq10}.  

In 
this theorem, we 
often interpret $X \subseteq \call_X(X)$ 
as the two-sided ideal of left-multiplication operators for any algebra $X$ 
without saying.  

Turn $X$ to a $\bfd$-bimodule via $\varphi$ 
by $x* d := x \cdot \varphi(d)$ 
and $d * x := \varphi(d) \cdot x$ for all $x \in X, d \in \bfd$.

{\bf Line 2:} 

Lemma \ref{lemma31} yields 
$e^{-1} \cdot s_+ \na_{s_-} = u_+ \na_{u_-}$
(in words: $e^{-1}$  multiplied by   line 1  is line 2).  
In the above diagram, $k:= \id_A$ and $l,m$ are the canonical
corner embeddings.

{\bf (a)} 
Set the algebra homomorphism $\varphi_\infty$  
to be the 
two-fold $G$-equivariant 
composition - where 
the `starter' algebra $ (\bfd, \delta) \otimes (M_n, \sigma)$  
of the second line of the above diagram 
has the indicated 
very special $G$-action 
by assumption -  
\be{eq67}
\xymatrix{ 
\varphi_\infty: 
\big (M_n(\bfd), \delta \otimes 
 \sigma  
 \big )   \ar[rr]^-{ \varphi  \otimes \id_{M_n}  }   &&  
 \Big ( \big ( \call_X(X) \big ) \otimes M_n,  \ad(\gamma_\pm)  \otimes      
   \sigma  
   \Big ) 
\ar[d]^-\tau   \\ 
& & 
\call_{M_n (X)} \Big ( \big (M_n (X),   \mu_\pm 
 \big ) \Big ),} 
\en
the last inclusion $\tau$ 
being obviously defined using matrix multiplication
\big ($\tau(y)(x):= y \cdot x$ 
for $y \in M_n\big (\call_X(X) \big), x \in M_n(X)$\big).

   We regard $M_n(\bfd)$ and 
   $M_n(X)$ as $M_n (\bfd)$-bimodules, 
the latter one  by left and right multiplication with the elements 
 $\varphi_\infty  (a)
 \in \call M_n(X)$  
 for all $a \in M_n(\bfd)$.

{\bf (b)}     
    Using that $s_\pm$ are $\bfd$-bimodule maps, 
   it follows that 
   $u_\pm$ as entered in the second line of the above
   diagram are 
   $M_n( \bfd)$-bimodule 
   maps, for example, 
      $$u_-( a  \cdot b ) = u_- \big ( (a_{ij})_{ij} \cdot (b_{ij})_{ij} \big ) 
      = \Big (\sum_k s_-(a_{ik} \cdot b_{kj}) \Big )_{ij} 
   =  \Big (\sum_k s_-(a_{ik}) * b_{kj}  \Big )_{ij}  $$
   $$= \Big (\sum_k s_-(a_{ik}) \cdot \varphi (b_{kj}) \Big )_{ij} 
   = u_- (a) \cdot \varphi_\infty(b)  = u_- (a) * b  . $$

   {\bf Line 3:}

   Recall $\bfx = 
   \jmath 
   (\bfj)+ S_-(\bfk)$ (see Subsection \ref{subsec28})  and other notions used 
   next from the above diagram. 

Define the algebra homomorphism  
\be{eq128}
\vartheta:  (\bfk, \kappa)  \rightarrow 
 \big ( \call_\bfx ( \bfx) , \ad(\Omega_\pm) \big) :  
 \vartheta (k) := \varphi_\infty \big (\phi(k) \big ) \oplus  k  ,  
 \en 
  at first 
  glance an element of $\call \big ( M_n(X) \big )   \oplus  \bfk$.  

Turn $\bfx$ to a 
$\bfk$-bimodule via $\vartheta$ 
\big (i.e. $x * k:= x \cdot \vartheta(k)$
and $k * x:=  \vartheta(k) \cdot x$  for all $x \in \bfx,k \in \bfk$\big ).  
It follows that $S_\pm$ are 
$\bfk$-bimodule maps, for example, 
$$S_\pm ( k \cdot w) = u_\pm \big (\phi(k 
 \cdot w)  \big ) \oplus k w 
 = \varphi_\infty   \big ( \phi(k) \big )  \cdot   u_\pm \big (\phi(w) \big ) \oplus k w
 $$
$$ =   \vartheta  (k) \cdot  
S_\pm  (w) = 
k * S_\pm(w)  .   $$

This computation 
also shows
that  $\vartheta$ lands in $\call_\bfx (\bfx)$.  
(The 
summand $\id_B$ was added 
in the formulas of $S_\pm$ to ensure they 
become  injective.) 

Immediately verify with 
Subsection \ref{subsec24} that $\phi \cdot u_+ \na_{u_-}
= S_+ \na_{S_-}$ 
(in words: $\phi$  multiplied by   line 3  is line 2).

   {\bf Lines 4 \& 7:} 
   
   We are going to show that $\Delta_{t_-} \cdot  S_+ \na_{S_-} = (S_+ \oplus \id_B) 
\na_{S_- \oplus \id_B}$ 
(in words: $\Delta_{t_-}$  multiplied by   line 3  is line 4). 

   {\bf (c)}   
   The particular shape of the third line of the diagram 
   is now unimportant, we only assume it is a level-one morphism
   $S_+ \na_{S_+}$ provided with a two-fold equivariant ($+$ and $-$) algebra homomorphism 
   $\vartheta: (\bfk, \kappa)  \rightarrow \big (\call_\bfx(\bfx), \ad(  \Omega_\pm  
   ) \big)$ 
   with respect to which $S_\pm$ are $\bfk$-bimodule 
   maps \big (i.e. $S_\pm(a \cdot b) = S_\pm(a) \cdot \vartheta(b)
   =  \vartheta(a) \cdot S_\pm(b)$ for all $a,b \in \bfk$\big ). 
   We now extend it 
   to  
   \be{eq129}
   \overline \vartheta: \big (\call_\bfk(\bfk), \ad(\kappa) \big )   \rightarrow 
   \big ( \call_\bfx(\bfx),
    \ad(  \Omega_\pm  
    ) \big)    
    \en   
   by  Lemma  \ref{lemma22}.(i).  
   We recall and will use Lemma \ref{lemma22}.(ii) that $S_\pm$ 
   are also $\call_\bfk(\bfk)$-bimodule maps 
   with respect to $\overline \vartheta$.   
   We define the algebra homomorphism 
      \be{eq59}
      T_- : (B, \beta) \rightarrow \big ( \call_{\bfx}(\bfx), \ad(         
      \Omega_\pm  
       \big )  : 
   	T_- := \overline 
   	\vartheta \circ t_-   .  
   	\en

{\bf (d)}    
   It follows that $S_\pm \oplus \id_B$, defined by   
   $S_\pm \oplus \id_B(k \oplus b):= S_\pm(k) \oplus b$
   for all $k \in \bfk, 
   b \in B$, are algebra homomorphisms,
   by the following and further similar computations 
   using Lemma \ref{lemma22}.(ii): 
 $$(S_- \oplus \id_B) \big ((a \oplus 0) \cdot (0 \oplus  b)  \big  ) =  (  
 S_- \oplus \id_B) \big (  a \cdot t_-(b)  \oplus  0 \big ) 
 = 
 S_-(a) * t_-(b)  \oplus 0  $$ 
 $$= S_-(a) \cdot \overline 
 \vartheta \big (t_-(b) \big )  \oplus 0  = S_-(a) \cdot T_-(b)  \oplus 0  $$
 $$= \big (S_-(a) \oplus 0 \big ) \cdot (0 \oplus b)  = (S_- \oplus \id_B)
 \big (a \oplus 0 \big ) \cdot  (S_- \oplus \id_B) \big (0  \oplus b \big ) 
 . $$

{\bf (e)}		
Ignore the fifth line of the above diagram for the moment, and identify the fourth 
and the sixth line of the above diagram.  

Then line 7 
is just a copy of line 4 (= line 6), where only the two indicated algebras are transformed by  the 
algebra isomorphisms  
$\zeta_2(x \oplus b) := x + T_-(b) \oplus b$ 
and 
$\zeta_1(k \oplus b) := k + t_-(b) \oplus b$ 
as explained in  \re{zeta}.  

{\bf (f)}		
The synthetic 
split $\Delta_{t_-}$ showing up in the last 
column of the above diagram is then understood to be 
defined as the synthetic 
split $\Delta_\cals =: \Delta_{t_-}$ 
of the splitexact sequence 
\begin{equation}   \label{eq8b}
\xymatrix{ 
\cals: 
0   \ar[r]     & \bfk     
 	\ar[rr]_{\id_\bfk \oplus 0_B}  &  & \bfk 
\stackrel{t_-}{\oplus}   B      \ar@<.5ex>[rr]^{ (0 \oplus 
\id_B)  \circ \zeta_1 }    
 	\ar@<-.5ex>[ll]_{ \Delta_{t_-}}   
& &   B \ar[ll]^{ \zeta_1^{-1} \circ (t_- \oplus \id_B) }   \ar[r]   
& 0        
}
\end{equation} 
found 
embedded in the last column of the above diagram. 

We notice that 
the homomorphisms of \re{eq8b} 
are actually understood to point to the algebra  $\bfk 
\stackrel{t_-}{\oplus}   B$ of the {\em fourth} line
 of the above diagram, and not to the 
sixth line. 
This simple 
` flaw' in the diagram of the 
theorem should make 
the diagram appear simpler.

\big (However, $\zeta_1^{-1} \circ (t_+ \oplus \id_B)$ 
points correctly to the {\em sixth} line. 
In other words, 
$t_\pm$  
map into the different corner algebras of the `starter' 
algebra 
$M_2 \big ( \bfk 
\stackrel{t_-}{\oplus}   
  B    \big ) $ 
of the {\em fifth} line of the above diagram
  according to \re{eq10}.\big ) 
 

{\bf (g)}			
By obvious commutativity of 
the third and fourth line of the above diagram, 
we get 
$(\id_\bfk \oplus 0) \cdot (S_+ \oplus \id_B) 
\na_{S_- \oplus \id_B} = S_+ \na_{S_-}$  
(in words: $\id_\bfk \oplus 0$  multiplied by   line 4  is line 3)
and thus, 
by the splitexactness axiom 
of  
Subsection \ref{subsec22}.(r)  
for $\cals$ at the middle space
$\bfk 
\stackrel{t_-}{\oplus}   
  B$  of \re{eq8b},  
  that  
$$\Delta_{t_-} \cdot  S_+ \na_{S_-} =  
\Delta_{t_-} \cdot (\id_\bfk \oplus 0) \cdot (S_+ \oplus \id_B) 
\cdot  
\na_{S_- \oplus \id_B}$$
$$= \big (1- \zeta_1  
(0 \oplus \id_B) (t_- \oplus \id_B) \zeta_1^{-1}  \big ) 
\cdot (S_+ \oplus \id_B) 
 \cdot \na_{S_- \oplus \id_B}$$
$$= (1- 0_\bfk \oplus \id_B  \big ) 
\cdot (S_+ \oplus \id_B) 
\na_{S_- \oplus \id_B} = (S_+ \oplus \id_B) 
\na_{S_- \oplus \id_B}  , $$
the last identity being because $(0 \oplus \id_B) \cdot 
(S_+ \oplus \id_B) = (0 \oplus \id_B) \cdot 
(S_- \oplus \id_B)$ and $s \na_s = 0$ holds in general 
by \cite[Remark 3.6.(iii)]{gk}. 
This shows the claim. 

  {\bf Line 5:} 
  Supposing 
  the $M_2$-action of line 4 
  of the above diagram is special
  (what we shall show), we can apply 
  Lemma \ref{lemma31}		
  and obtain $e_1^{-1} \cdot (S_+ \oplus \id_B) \na_{S_- \oplus \id_B} 
   = v_+ \na_{v_-}$ 
   (in words: $e_1^{-1}$  multiplied by   line 4  is line 5).  
   The arrows $E_1,E_2$ in the above diagram are the canonical 
   corner embeddings.  

  {\bf Line 6:} 
  By commuting diagrams, Subsection \ref{subsec24}, we get 
$e_2 \cdot  v_+ \na_{v_-} = (S_+ \oplus \id_B) \na_{S_- \oplus \id_B} $
(in words: $e_2$  multiplied by   line 5  is line 6). 

{\bf Line 8:}  

If $z_+ \na_{z_-}$ labels the 
seventh line of the above diagram, one checks with 
Subsection \ref{subsec24} that $(t_+ \oplus \id_B) \cdot z_+ \na_{z_-} = x_+ 
\na_{x_-}$ (in words:  
${t_+ \oplus \id_B}$  multiplied by   line 7  is line 8). 
   Thereby, the precise meaning of $x_\pm$ 
   is 
   \be{eq72}
   x_\pm(b)   := 
    z_\pm \circ (t_+ \oplus \id_B) (b) =   
 \big ( S_\pm  \circ (t_+ - t_-) (b)  + T_- (b)  \big )  \oplus b 
 \en  
 for all $b \in B$.  

The middle up-arrow $\id$ starting from line 8 is valid, 
as $S_\pm \big ((t_+ - t_-)(b) \big ) \in S_\pm (\bfk) \in \bfx$, which is the 
defining ideal of $\call \bfx   
   \square_{T_-}      B$ (see Subsection \ref{subsec28}), because it is the isomorphic image of $\zeta_2$ (see \re{zeta}), 
   and thus - by \re{middle} and the definition 
   that $\jmath (\bfj)$ 
   is the defining ideal of 
   $\call  \bfx   
\square_{x_-} 
  B$ -
   $$\jmath (\bfj) \oplus 0_B + x_-(b)  \in \bfx \oplus 0_B + T_-(b) \oplus b \in  \call \bfx   
   \square_{T_-}      B   . $$ 

Note that 
$\jmath  
(\bfj)$ is indeed an ideal in $\call_\bfx (\bfx)$, 
because it is one in $\bfx$ \big (because of quadratik algebras: 
$j_1 j_2 \in \jmath (\bfj)$ for $j_1,j_2 \in \jmath (\bfj)$, thus $j_1 j_2 \cdot \call \bfx  
\subseteq j_1 \cdot  \bfx 
\subseteq \jmath (\bfj)$\big ).

   {\bf $M_2$-spaces:} 
   
   We are going to define the $M_2$-spaces of all lines, 
   and must also verify that the maps $s_-$ and $s_+$ 
   of a line $s_+ \na_{s_-}$ map equivariantly into the upper 
   left and lower right corner, respectively,  of the $M_2$-space 
   of that line. 
   
   Our ultimate goal is to define the $M_2$-action of the fifth 
   line of the above diagram, because all other spaces 
   of the above diagram map injectively into that $M_2$-space 
   via some algebra homomorphisms,
   and these injections must be equivariant. 

      {\bf $M_2$-space of Line 2:}   

   Recall that $ \big (M_2(X), \ad(\gamma_- \oplus \gamma_+) \big)$ is the $M_2$-space of the first line of the above diagram. 
   
   {\bf (m)}
   
   The $M_2$-space of the second line of the above diagram is determined by  
   Lemma \ref{lemma31}, and since 
   $e:(\bfd,\delta) \rightarrow (\bfd, \delta) \otimes (M_n, \sigma)$  
   is supposed to be very special, 
   we obtain the $M_2$-space of the second line of the above
   diagram to be the 
   $G$-invariant subalgebra 
       \be{eq120}
   M_2 \big (   M_n(X) \big )  
   \subseteq 
   \Big (M_2 \big ( \call_{ M_n(X)} \big ( M_n(X) \big  )\big ), \ad( 
   \mu_- \oplus \mu_+  
   ) \Big ) , 
   \en
   $$ \mu_\pm := \gamma_\pm \otimes \sigma$$ 
	with 
	$\mu_+$ 
	being a right $G$-module action on the $G$-algebra
	$(X \otimes M_n, 
	\mu_- )$.  

	{\bf (n)}		
	Also as stated by Lemma \ref{lemma31}, the $M_2$-space \re{eq120} 
	is special if $n$ is finite, and even very special if 
	the $M_2$-space of $s_+ \na_{s_-}$ is very special. 
	These assumptions are made in the statement of the theorem.  

	{\bf (o)}    	
	Ensured by Lemma \ref{lemma31}, 
	  the algebra homomorphisms $u_-,u_+$ of the second line of the above diagram must be $G$-equivariant 
	  maps into the upper left and lower right corner embeddings,
	  respectively, of 
	  the $M_2$-space \re{eq120}, that is,  
	  with respect to $\ad(\mu_-)$ and $\ad(\mu_+)$ 
	  (as multiplication operators), 
	  or simply with respect to $\mu_-$ and $\mu_+$ 
	  (directly).  
	
      {\bf $M_2$-space of Line 3:} 

   {\bf (p)}			
   Recall from the diagram that $\bfx:= (M_n(X) \square_{S_-} \bfk $. 
   Let $(\bfk, \kappa)$ be the `starter' algebra of line 3. 
   The $M_2$-space of the third line of the above diagram is set to be  
      \be{eq125}
       \Big ( M_2(\bfx), \ad(   \Omega_- \oplus \Omega_+  
       )   \Big ) := 
       \Big ( M_2 \big (M_n(X) \square_{S_-} \bfk \big )  ,  
   \ad \big (( \mu_- 
   \oplus  \kappa )  |_\bfx  
   \oplus 
   (   \mu_+   
   \oplus \kappa )  |_\bfx  \big )   
	\Big ) 
	, 
	\en
	$$ \Omega_\pm := ( \mu_\pm 
   \oplus  \kappa )  |_\bfx  . $$ 
   
   Here, $\mu_+  
   \oplus \kappa$ is a right $G$-module 
   action on the $G$-algebra $ \big (M_n(X) \oplus \bfk,
    \mu_-  
    \oplus \kappa \big )$.

    {\bf (q)}			
	The $G$-actions 
   $\mu_\pm   
   \oplus \kappa$ 
   are indeed invariant on $\bfx$, 
   which we may write both ($+$ and $-$) as  
   $\bfx = \iota_n(\bfj) \oplus 0_\bfk + S_\pm(\bfk)$ 
   by \re{middle} and \re{middle2},  
    since 
    $$(  \mu_{\pm,g}   
    \oplus \kappa_g) \big (\iota_n(j) \oplus 0 +  S_\pm(k) 
    \big) 
    = \mu_{\pm, g} 
    \big(\iota_n(j) \big) \oplus 0 + S_\pm \big (\delta(k)  \big ) 
    \in  
   \bfx$$ 
   for all $j \in \bfj, k 
   \in \bfk$, 
   since the ideal $M_2 \big (\iota_n (\bfj) \big)$ is \big (always by Subsection
   \ref{subsec11}.(g)\big ) invariant under the $M_2$-space of the second line of the above diagram, and by item 
   (o) above.

   {\bf (r)}			
   Since \re{eq120}, that is the $M_2$-space of the second line of the above diagram, is special by (n), we have  by Lemma \ref{lemma32}  
   and the formula \re{eq300} for the $G$-action $\delta$ there 
   (to be evaluated in each $g \in G$ - not notated) 
 \be{eq122}
   \Omega_- \circ \Omega_+^{-1}   - \Omega_- \circ \Omega_-^{-1}  
   =  (\mu_- \circ \mu_+^{-1}   \oplus  \id_\bfk 
   -  \mu_- \circ \mu_+^{-1}  
   \oplus \id_\bfk 
   )|_\bfx  \; \in  \;   
     \jmath (\bfj)   .  
     \en 

 That is, without the restriction to $\bfx$ this is a multiplication operator in 
 $\jmath (\bfj) \subseteq M_n(X) \oplus \bfk$, so that 
 its restriction to $\bfx$ is also such an multiplication
 operator, since 
 $\jmath (\bfj) $   is an ideal in $\bfx$. 
 
 Hence, \re{eq125} defines a special action on $M_2(\bfx)$ by Lemma \ref{lemma32} and \re{eq122}. 

{\bf Equivariance of $\vartheta, \ol \vartheta, T_-$:} 

We have defined 
$\mu_\pm $  
 in \re{eq120} 
 as anticipated in \re{eq67}, and thus  
 $\varphi_\infty $ is equivariant with respect to both $G$-actions 
 $\pm$ 
 notated  
 in \re{eq67}, and hence $\vartheta$ and $\overline \vartheta$, respectively, - see \re{eq128}  and \re{eq129} - are equivariant with respect to 
 (both $+$ and $-$ simultaneously) 
 $\Omega_\pm$ and $\ad(\Omega_\pm)$, respectively.  
 
 
 Consequently, 
 as $t_-$ is equivariant by assumption, $T_-$ is equivariant 
 with respect to 
 (both $+$ and $-$) $\ad(\Omega_\pm)$,    
 respectively, see \re{eq59}.

      {\bf $M_2$-space of Line 4:} 
    
    We cannot canonically define the $M_2$-action
    of the fourth line of the above diagram, but of the 
    non-equivariantly isomorphic seventh line of the above diagram:
    
    {\bf (s)}				
    The $M_2$-space of the {\em seventh} line of the above diagram
    is {\em preliminary} set to be  (recall 
    $\Omega_\pm$  from \re{eq125})
                  \be{eq123}
        \Big ( M_2  \big (
    \call_\bfx 
    (\bfx 
    ) \; 
    \square_{T_-}   B \big ) 
    \subseteq 
    \Big ( M_2 \big (\call_{\bfx \oplus B}(\bfx \oplus B)  
     \big ) ,
        \ad  ( 
        \Omega_- \oplus \beta   \oplus 
       \Omega_+  
       \oplus \beta) 
    \Big )   . 
    \en     
    
    Applying this inclusion, 
    \re{eq122} and formula \re{eq300}  
    to lemma  \ref{lemma32} 
    \big (and note that by \re{middle} the ``ideal $J$''  
    of that lemma is $\bfx \oplus 0_B \cong 
    \calk_{\bfx}(\bfx)$ 
    and thus $G$-invariant as always compact operators are, and the ``subalgebra $A$'' of that lemma is 
    $\{T_-(b) \oplus b | b \in B\}$, which is also $G$-invariant 
    under the left top corner action as
    $T_-$ is equivariant for $\ad(\Omega_-)$\big ),
    we see that the left hand side algebra of \re{eq123} is 
    indeed a $G$-invariant subalgebra  with a special $G$-action.

	{\bf (t)}			

	By {\em assumption}, 
	the `starter' algebra 	
	of the {\em fourth} line of the above diagram  
	is  $\big (\bfk \stackrel{t_-}{\oplus} B,  \kappa \oplus \beta \big )$, 
	say, and 
	applying to this space the isomorphism $\zeta$ 
	of Subsection \ref{subsec28} we get the 
	space $\big (\call \bfk \square_{t_-} B, 
	\ad (\kappa) 
	\oplus \beta \big )$ (ideal is $\bfk$), 
	what we assume as the definition for the `starter' algebra of the {\em seventh}  line of the above
	diagram. 
	
	Because of 
	proven equivariance of $S_\pm$ and 
	$T_-$, both algebra homomorphisms 
	notated in the {\em seventh} line of the above diagram, 
	formerly denoted by
	$z_-$ and $z_+$ in 
	proof step ``Line 8'', are {\em equivariant} 
	 maps into the upper left and lower right corner of 
	the $M_2$-space  \re{eq123}, respectively.
	
	Now we {\em define} the $G$-actions of the sixth line of the above diagram including 
	its $M_2$-space in such a way that  
	the algebra isomorphisms $\zeta_2 \otimes \id_{M_2}$  and $\zeta_1$ between line 6 and 7, respectively, become $G$-equivariant. 
	
	Finally, we use this definition of (the $G$-actions of) 
	the sixth line of the above diagram as a definition for the {\em fourth} line
	of the above diagram . 
	 (Observe that the $G$-actions of its `starter' algebra 
	 has not changed, so is $\kappa \oplus \beta$ again.) 
	 
	      {\bf $M_2$-spaces of Lines 5-8:} 

		
	{\bf (u)}			
	We have seen from the preliminary seventh line that 
	it has special $M_2$-space, see 
	(s), which now has by `copy-paste' 
	up to isomorphism also the sixth and fourth line. 
  
  	Hence, 
   the $M_2$-action of the fifth line of the above diagram 
   is again determined by Lemma 
   \ref{lemma31} and the fourth line.

   {\bf (v)}					
   The $M_2$-action 
   of the sixth line 
   is defined  such that the injection 
   (corner embedding) $E_2 \otimes \id_{M_2}$ becomes equivariant. 
   
   The $M_2$-space of the 
   seventh line of the diagram is determined by the 
   sixth line  
such that $\zeta_2 \otimes \id_{M_2}$ becomes equivariant, 
and it is restricted to define the  invariant $M_2$-space of the 
eigth line. 

 {\bf Final:} 
 Writing down $t_+ \na_{t_-}$ as in 
 \re{eqt} and then multiplying it with $s_+ \na_{s_-}$, then successively the factors of $t_+ \na_{t_-}$ 
 are absorbed with every step we go down at the diagram,
 that is, $e^{-1} \cdot s_+ \na_{s_-} = u_+ \na_{u_-}$ (line 2),
 $\phi \cdot u_+ \na_{u_-} = S_+ \na_{S_-}$ (line 3)  
 et cetera, until we end up with $x_+ \na_{x_-}$, 
 which is the claim of the proposition. 
  
  {\bf Formula:} 
  The indicated formula of $x_\pm$ comes out 
  by starting from \re{eq72} and successively entering 
  the defined maps.

  {\bf Homotopy.} 
  Observe that we nowhere used homotopies, also not in the 
  lemmas of Section \ref{sec3} 
  or in 
  Subsection \ref{subsec24} 
  about commuting diagrams. 
  
  {\bf Center.}
  We show that if $\varphi$ maps into the center and is equivariant
  with respect to $\ad(\gamma_-)$ so with respect to $\ad(\gamma_+)$.
  Let $T:= \varphi(d)$. Then $\gamma_{-,g} \circ T \circ \gamma_{-,g}^{-1}
  = \gamma_{+,g} \circ T \circ \gamma_{+,g}^{-1}$ if and only if 
  $T \circ \gamma_{-,g}^{-1} \circ \gamma_{+,g}  
  = \gamma_{-,g}^{-1} \circ \gamma_{+,g} \circ T$. 
%
 \end{proof}

As a  corollary of the last theorem,  
we can now deduce 
the computation of 
the $K$-homology $K$-theory product: 

 \begin{corollary}[$K$-homology $K$-theory product]
 			\label{cor31}
 
 If $\bfc$ is unital and $G$ 
 a unital inverse semigroup, 
 in very special 
 $GK^G$-theory 
 (in particular for $\bfc:= \C$ if $G$ is a group) 
 the last theorem applies, that is, we 
 have the 
 explicitly computable product map  
$$L_1 GK^G (A,\bfc)  \otimes_\Z  L_1 GK^G (\bfc, B) 
 \rightarrow L_1 GK^G (A,B) :  z \otimes w \mapsto z \cdot w  . $$
  
  This holds analogously for $C^*$-theory, that is,
  after (easy and explicit) back and forth translation $KK^G \cong GK^G$ \cite{aspects} (only very special $KK^G$-theory), the 
  intersection product 
  $$KK^G (A,\bfc)  \otimes_\Z  KK^G (\bfc, B) 
 \rightarrow KK^G (A,B) :  z \otimes w \mapsto z \otimes_\bfc w$$
	is explicitly computable by the formula indicated in 
	Theorem \ref{thm1}. 
	
 \end{corollary} 

\begin{proof}

We apply Theorem 
\ref{thm1}  to the case $\bfd:=\bfc$,
$\bfk :=M_n(\bfc)$, 
which has  
an obvious approximate unit $(p_n)$ of projections (as $\bfc$ is unital), 
and to $\phi :=\id_{M_n(\bfc)}$,  
and $\varphi: (\bfc,\chi) \rightarrow \big (\call_X(X),\ad(\gamma_\pm)\big )$ 
defined by $\varphi(c)(x):= x * c$ for $c \in \bfc,x \in X$ 
- that is, in other words, 
$\varphi(e) = (\gamma_-)_e \in \calz \call_X(X)$ for 
$e \in E$ -, 
as described in 
 Subsection \ref{sec22}.(c), and recall Subsection \ref{sec23}.(c). 
 Note that $\varphi$ is equivariant with respect to $\ad(\gamma_-)$ 
 and thus with respect to $\ad(\gamma_+)$ by 
 the theorem as 
 $\varphi$  maps into the center.
 
 Also, if all inverted 
 corner embeddings in $s_+ \na_{s_-}$ and $t_+ \na_{t_-}$ 
 are very special, then $\gamma_+ = \gamma_-$ and finally 
 $\mu_+ = \mu_-$ in \re{eq123} and the whole diagram of the theorem
 consists of very special inverted 
 corner embeddings only, confer Lemma 
  \ref{lemma31} . 
  	%
\end{proof}

 Actually we have  \re{eq33} below, 
 which formally expands the domain of the above product function. 
 For $\bfc :=\C$ 
 this was proven in \cite{gk2},  but 
 the proof 
 there seems to go through without any essential 
 modification also for $\bfc$-compatible $GK^G$-theory 
 by verbatim the same proof: 
 
 \begin{lemma}[Conjecture]		\label{cor32}

If $G$ is an inverse semigroup and $\bfc$ and $G$ 
are unital, 
then  
all results of \cite{gk2}    
hold also true in $\bfc$-compatible, inverse semigroup equivariant   $GK^G$-theory 
if $\C$ is replaced by $\bfc$ there. 
In particular, as sets we have 
by \cite[Corollary 3.5]{gk2}  
that 
\be{eq33}
L_1 GK^G (\bfc, B) = GK^G(\bfc, B)    . 
\en 
  

\end{lemma} 

 \section{Products with synthetic 
 splits}
					\label{sec5}

%
 

This lemma shows how to potentially compute the product of a level-one element 
with a synthetic   
split at the first place:  

\begin{lemma}[Fusion with $\Delta_s$]			\label{lemma51}

Consider the split-exact sequence   \re{eq8}. 

{\rm (a)} 
If 
elements $u,v$ in $GK^G$
fulfill $u = 
\iota v$ ($v$ extends $u$),  
then $z :=  
(\id_X - g s ) v$ 
satisfies  $u = 
\iota z$ and  
$$\Delta_s u = z $$   
($u$ can be fused with $\Delta_s$). 
If $v$ is a level-one element, so is $z$.

{\rm (b)} 
Conversely, 
if elements $u,z$ satisfy  $\Delta_s u = z$,  then
$u = 
\iota z$.


\end{lemma}

\begin{proof}
(a) 
Indeed, by the splitexactness axiom of Subsection.\ref{subsec22}.(r) we get 
$$\Delta_s u = \Delta_s \iota v = (\id_X - gs) v =: z   . $$

If $v$ is a 
level-one element, then 
by \cite[Lemma 9.5]{gk}
we can fuse $v$ with 
$\id_X$ and $gs$, and sum it up 
to a level-one element by 
\cite[Corollary 9.10]{gk} and \cite[Lemma 9.9]{gk}. 

(b) 
By the splitexactness axiom of Subsection.\ref{subsec22}.(r) we 
get 
$u = \id_J 
u = \iota \Delta_s u =  \iota z$.  
\end{proof}

A special split exact sequence is the canonical one associated to unitization of an algebra $A$ (see Subsection \ref{subsec24}),
that is,  
 \big (here $s(c):= c \cdot1_{A^+}$, $f(a+c\cdot 1_{A^+}):= c$, $\iota(a):= a + 0 \cdot 1_{A^+}$ for all $a \in A, c \in \bfc$\big )  
$$  
\xymatrix{ 
0   \ar[r]     & A      
 	\ar[rr]_{\iota}  &  & A^+       \ar@<.5ex>[rr]^f    
 	\ar@<-.5ex>[ll]_{ \Delta_{A}}   
& &   
\bfc \ar[ll]^{s}   \ar[r]   
& 0        
}
$$		


\begin{corollary}			\label{cor5}

Any 
level-one 
element $x \in L_1$  
with special
$M_2$-space can be fused 
with the elementary split $\Delta_{A}$ 
to 
a level-one element 
$\Delta_A \cdot x \in L_1$ 
(provided composability). 

\end{corollary}

\begin{proof}
The element $s \Delta_{s}$ 
(which is actually zero by \cite[Remark 3.6.(iii)]{gk}) 
associated to the above diagram is in $L_1 GK^G(\bfc,A)$ 
with 
very special $M_2$-space, so can 
be fused by \cite[Proposition 3.4]{gk2}    
to an element $s \Delta_{s} 
\cdot x \in L^1$. We 
halt the proof however already at the stage 
$\Delta_{s} 
\cdot x \in L^1$, yielding the result.
\end{proof}

\section{Compare of $kk$-theory with $GK$-theory} 
						\label{sec6}


In \cite{cuntz}, Cuntz developed the universal diffeotopy-invariant, 
matrix-stable (where a matrix algebra is closed under a certain 
`smooth' topology) 
and linear-split half-exact 
(i.e. each short exact sequence in algebra with a  linear split  
induces a half-exact sequence in theory) 
theory for 
locally convex algebras induced by submultiplicative 
seminorms, and 
a core construction of its morphisms 
is an exact finite sequence (exact in $\iota$) 
\be{eq160}
\xymatrix{0 \ar[r] & J_1 \ar[r]_{\iota_1}  & J_2 \ar[r]_{\iota_2}   
\ar@<-.5ex>[l]_{t_1}   & 
\ar@<-.5ex>[l]_{t_2}  J_3     \ar[r]_{\iota_3}  & 
\ar@<-.5ex>[l]_{t_3} 
\cdots \ar[r]_{\iota_n} & J_n \ar@<-.5ex>[l]_{t_n}  \ar[r] &   0}
\en 
of algebras $J_i$ and continuous algebra homomorphisms $\iota_i$ 
with continuous {\em linear splits} $t_i$, see \cite[Definition-Satz 3.6]{cuntz}. 
In lemma \ref{lemma62} below it is remarked that the universal splitexact theory 
$GK^G$-theory can also brought to a 
remotely similar pattern, but now with synthetic 
generator splits $\Delta_s$.
%

This $kk$-theory was modified in \cite{cortinasthom} 
by 
Corti\~nas and Thom 
to obtain $kk$-theory for the class of algebras or rings, 
which again was extended to group equivariant $kk^G$-theory 
(with `very special' $G$-actions) by Ellis in \cite{ellis}. 
Related are also Cuntz and Thom \cite{cuntzthom},
Garkusha \cite{garkusha1,garkusha2}, Weidner \cite{weidner, weidner2} 
and Grensing \cite{grensing}.  
On the other hand, $GK^G$-theory \cite{gk} denotes the 
universal 
stable, (possibly polynomial or smooth) homotopy invariant
and {\em split-exact} theory for rings or algebras. 

Even if we partially deviate 
next from the considered class 
of algebras and rings 
for $GK^G$-theory 
than in \cite{gk} 
(note that we shall only use its definition, not any results, 
and 
stability and homotopy invariance 
 only as far as implemented in the respective theory to be compared), we  
might safely remark 
and partially anticipate:

 \begin{lemma} 
				\label{lemma61}

Consider the universal homotopy invariant, matrix stable and linear-split half-exact $kk$-theory $kk$ 
by Cuntz \cite{cuntz}
constructed for the category of 
locally convex algebras induced by submultiplicative seminorms 
\cite{cuntz}, locally convex algebras 
by Cuntz \cite{cuntzweyl}, 
algebras 
or rings  by 
Corti\~nas and Thom \cite{cortinasthom}, 
and very special $G$-equivariant algebras 
or rings by Ellis \cite{ellis} 
denoted $kk^G$, respectively, 
and their analogous universal homotopy invariant, matrix stable and  splitexact counterparts 
very special $GK^G$-theory $GK^G$ \cite{gk}. 

Then there exists a functor (here $f$ 
is any algebra homomorphisms) 
$$\Gamma: GK^G \rightarrow kk^G :\Gamma(f) = f .$$

\end{lemma}

\begin{proof}

We follow the proof of Higson \cite[Theorem 3.5]{higson2} (the $KK \rightarrow E$ 
 case in $C^*$-theory). 

{\bf (a)} 
In \cite[p. 269]{higson} Higson remarks, that for a functor 
$F:R \rightarrow \bfa$ from a class of algebras $R$ 
into an 
additive category $\bfa$ 
the functor $F(A,-)$ is stable, homotopy invariant and 
splitexact for each object $A\in R$ if and only if $F$ itself has these properties, 
which means $\bfa$ is provided with morphisms $F(e)^{-1}s$ 
(inverses of  {\em matrix} corner embeddings $e$) and $\Delta_\cals$s (synthetic  
splits for short splitexact sequences $\cals$, see \re{eq8}, 
in the image of $F$) and $F(f)$s (where $f$ are algebra homomorphisms)  
as in $GK^G$-theory 
defined in \cite[Definition 3.5]{gk}
such that the 
defining axioms 
of $GK^G$-theory hold in $\bfa$ for the morphisms
$\{ F(f), F(e)^{-1} , \Delta_\cals|\, f, e ,\cals\}$. 
Hence, by the universal construction of $GK^G$-theory 
by generators and relations \cite[Definition 3.5]{gk},
we are provided with a canonical functor $GK^G \rightarrow \bfa$ 
which sends the generators of $GK^G$ to 
the 
aforementioned morphisms of $\bfa$
and the product and addition of $GK^G$ to those of $\bfa$.

{\bf (b)}
In both Cuntz' $kk$-theories 
for locally convex algebras \cite{cuntz} and \cite{cuntzweyl}, 
respectively,   
a given short splitexact sequence \re{eq8} 
becomes a 
short splitexact 
sequence in the image of the functor $F(-):=kk(A,-)$,   because 
by the cyclic six-term exact sequences of 
\cite[Theorem 5.5]{cuntz} 
and \cite[Theorem 8.5]{cuntzweyl}, respectively, the connecting maps 
$F(\partial)$ must be null since $F(f)$ is surjective by the split
$F(s)$, and thus $F(\iota)$ injective. 
Hence 
last item (a) applied to the functor $kk:R \rightarrow \bfa =: kk$ 
(the category $kk$) yields the desired functor 
$\Gamma: GK 
\rightarrow kk$. 

{\bf (c)} 
 The functors $kk(A,-)$ are also splitexact in the case of 
 category of algebras or rings by 
 Corti\~nas and Thom \cite{cortinasthom} 
 by \cite[Corollary 6.4]{cortinasthom}, 
 yielding the claim also in that case by (a). 
{\bf (d)} 
Ellis \cite{ellis} defines $G$-equivariant $kk$-theory for algebras 
as  
$$kk^G \big ((A,\alpha), (B,\beta) \big):= kk_{G-{\rm Alg}} \big (
(M_{|G|} \otimes A, \lambda \otimes \alpha), 
(M_{|G|} \otimes B, \lambda \otimes \beta) \big ),$$
where $\lambda_x(e_{g,h}) = e_{xg,xh}$ 
for all $x,g,h \in G$ and $kk_{G-{\rm Alg}}$ is 
Corti\~nas and Thom's 
{non-equivariant} $kk$-theory formally restricted to the $G$-algebras. 
 By  
 definition of $kk^G$, 
 equivariant algebra homomorphisms $f:A \rightarrow B$ go to algebra homomorphisms $\id_{M_{|G|}} \otimes f$, whence 
 splitexactness of $kk^G(A,-)= kk_{G-{\rm Alg}}( 
 M_{|G|} \otimes A, \id_{M_{|G|}} \otimes -)$ 		
 follows from 
 an analogous proof of  
 \cite[Corollary 6.4]{cortinasthom} again 
 by equipping the auxiliary spaces 
 appearing there with the trivial $G$-action   - 
 this is {\em conjectural}, but 
 used excision in that proof is verified by 
 Ellis in $kk^G$ \cite{ellis} -, and thus 
 we have a functor $\Gamma: GK^G \rightarrow kk^G $ 
 by (a) again. 
  %
%
\end{proof}

Due to $\Gamma$, computations one achieves in $GK$-theory 
immediately 
descend to $kk$-theory, where one may exploit 
linear-split half-exactness, together with the fact that $\Gamma$ is 
faithful 
on $K$-theory. 
%
Indeed, we recall 
that besides the considered class of algebras, 
it turns out that the {\em blueprint} 
of $K$-theory in 
$kk$-theory by Cuntz \cite{cuntz} and  $GK$-theory \cite{gk} 
both coincide with the $K$-theory 
$K$ of Phillips \cite{phillips} for Fréchet algebras  
by \cite[Theorem 7.4]{cuntz} and \cite[Theorem 4.3]{gk2}:
$$kk(\C,A) \cong K(A) \approx GK(\C,A) \qquad
\mbox{(sloppy)}	. $$ 

\begin{lemma}[$e^{-1}$ skips $\Delta_s$]   \label{lemma52} 

In very special $GK^G$-theory, 
let a splitexact sequence 
 and a canonical corner embedding $e_J$ 
as in the first line and the first column, respectively, of the 
diagram below 
be given ($e_M,e_A$ are also the canonical corner embeddings).    

Then $e_J^{-1} \cdot \Delta_{s} = \Delta_{s \otimes \id} \cdot e_M^{-1} $
as readily verified by this commuting diagram: 


$$\xymatrix{
 (J ,\beta)  \ar[d]^{e_J}   \ar[r]_i &  (M, \gamma)   
 \ar@<.5ex>[r]^f   
 \ar@<-.5ex>[l]_{\Delta_s}  
 \ar[d]^{e_M}   &  (A, \alpha)   \ar[l]^s  \ar[d]^{e_A}   \\ 
  (M_n \otimes J, \kappa \otimes \beta)     \ar[r]_{i \otimes \id}  &  
  (M_n \otimes M  , \kappa \otimes \gamma)   
  \ar@<-.5ex>[l]_{\Delta_{s \otimes \id}} 
  \ar@<.5ex>[r]^{f  \otimes \id}    &  (M_n \otimes A,  \kappa \otimes \alpha )    \ar[l]^{s  \otimes \id}    \\
}$$

 \end{lemma}

 As announced, 
 the next lemma shows 
 some similarity of $GK^G$-theory to $kk^G$-theory 
 in structure, because morphisms in $GK^G$-theory are shown to be presented 
 by a sequence 
 of  
 {\em synthetic splits}  
 as opposed to {\em linear splits} in  \re{eq160} 
 in $kk$-theory.

\begin{lemma}		\label{lemma62}

In very special $GK^G$-theory, 
every $z \in GK^G(A,B)$ can be written in the form
$$ z =  s_+  \Delta_{s_1} \Delta_{s_{2}}  \ldots 
\Delta_{s_n}  e^{-1} . $$ 

\end{lemma}

\begin{proof}

Since every $z$ can be written as a finite product of $L_1$-elements 
by \cite[Lemma 10.2]{gk},  
it is sufficient by induction on $n$, that the product of the  
above $z$ with an $L_1$-element $x := t_+ \na_{t_-}= t_+ \Delta_{t_-} f^{-1}$
(the last identity because in very special $GK^G$-theory 
the product $f_2 \cdot f_1^{-1}$ in \re{eq29} vanishes by an ordinary 
rotation homotopy), that is $x \cdot z$, has the 
above form again:
%
%
$$ x \cdot z = t_+ \Delta_{t_-} f^{-1}  \cdot s_+  \Delta_{s_1} \Delta_{s_{2}}  \ldots 
\Delta_{s_n}  e^{-1}  $$
$$= u_+ \Delta_{u_-} F^{-1}  \cdot   \Delta_{s_1} \Delta_{s_{2}}  \ldots 
\Delta_{s_n}  e^{-1}  
= u_+ \Delta_{u_-}   \cdot   \Delta_{s_1'} \Delta_{s_{2}'}  \ldots 
\Delta_{s_n'}  F_n^{-1} e^{-1}  $$
where in the second line we used \cite[Lemma 9.5]{gk}
(fusion of an $L_1$-element with a homomorphism), and then
$n$-fold lemma  \ref{lemma52}. 
	%
	%
	%
\end{proof}

\section{Review of computable products}   
				\label{sec7}

We are going to summarize the computable products  in 
$GK^G$-theory - 
meaning that the product of two $L_1$-element is an $L_1$-element again -, 
which are then 
calculable in 
$KK^G$-theory by the same means too,  
and also indirectly in $kk^G$-theory 
via 
the functor $\Gamma : GK^G \rightarrow kk^G$ of Lemma \ref{lemma61}.

\subsection{Computable products}			\label{subsec7}
The next table gives an overview which 
generators $x \in \Theta$ can be multiplied with an 
$L_1$-element $\bfz \in L_1$ 
to an $L_1$-element, indicated with Yes and No in the table. Note that every element in $GK^G$
may be written as a finite product $x_1 \cdot \ldots \cdot x_n$ of $L_1$-elements by  \cite[Lemma 10.2]{gk},  
so that 
except the product 
$\bfz \cdot \kappa$ everything in the table below extends to all 
$\bfz \in GK^G$. 

Here $\bfz \in L_1 GK^G(A,B)$, 
${\kappa} \in L_1 GK^G(\bfc, A)$, 
$\phi$ is an algebra  homomorphism, 
$e^{-1}$ the inverse of a corner embedding, 
and $\Delta_s$ a synthetic    
split.  
Moreover, v.s.  $GK^G$, s. $GK^G$ and  $GK^G$,
respectively,  means all 
involved corner embeddings which manifest inverted 
in the stated product are very special, special and
generalized, respectively.   
  The stated products are 
  assumed to be composable: 
  

\begin{longtable}[ht]{lllll}      

	Product  &  v.s.  $GK^G$  &  s.  $GK^G$  &  $GK^G$  
 & Reference 
\\
	$\phi \cdot \bfz $    & Y  & Y & Y     & 
		\cite[Lemma 7.2]{gk} \\  
		$ \bfz  \cdot \phi   $    & Y  & Y & Y     &  \cite[Lemma 9.5]{gk} \\ 
				$ e^{-1}  \cdot  \bfz    $    & Y  & Y  &  
				$\Leftrightarrow \exists u \in L_1: \; e \cdot u = \bfz$ 
								&    Lemma \ref{lemma31}   \\    
				$   \bfz \cdot e^{-1}    $    & Y  & Y  & Y    
				& \cite[Corollary 8.2]{gk}   \\ 

	$\Delta_A \cdot \bfz $  
	& Y  &   Y  &  N 
	&  Corollary \ref{cor5}   \\  
	$\kappa \cdot \bfz  $  & Y  & NY & N    &  
				\cite[Proposition 3.4]{gk2}   \\   
 	$\bfz \cdot \kappa  $ & Y  & NY   & N     
 	&  Theorem  \ref{thm1}    \\  
 	$\Delta_s \cdot \bfz  $ & N & N &  $\Leftrightarrow \exists u 
 	\in L_1: \; i \cdot u = \bfz$ 
 	&  
 	Lemma  \ref{lemma51}   \\    
 	
\end{longtable}

In the last line we see, 
for 
the fusion with a general split $\Delta_s$ 
no formula is available so far, 
but the last two lines before the last line 
cover the important special cases $\kappa \cdot \bfz$ and 
$\bfz \cdot \kappa$, 
where the algebra $\bfc$ is involved in the product, 
notably used in the Dirac dual-Dirac method \cite{kasparov1988}.
Beside these synthetic 
splits $\Delta_s$ all other problems 
in the table  from left to right are 
caused 
by the 
`inability' so far of forming the product $e^{-1} \cdot \bfz$  
of  a {\em generalized} inverse corner embedding 
$e^{-1}: \calk_A(\cale \oplus A) \rightarrow A$ with a level-one element  $\bfz$, and the need of unitization in computing $\kappa \cdot \bfz$. 
The NY answer in the table means it depends on further assumptions, 
sometimes requiring the corner embedding $e$ to be into a finitely 
sized matrix (then it works always). 

\section{Implications for $KK$-theory}		\label{sec8}

Recall from Subsection \ref{subseckk} that $KK^G$-theory 
can be equivalently described as $GK^G$-theory \cite{aspects}. 

	{\bf (a)}  
     That is why all computations indicated 
     in the above table can be verbatim done in
     $KK^G$-theory, if analogously restricted to 
     very special and special $G$-actions if indicated. 
      Actually, every Kasparov cycle in $KK^G$-theory is equivalent 
      to a cycle with special 
      $M_2$-action, 
      by Lemma \ref{lemma32} and  
      \cite[Remarks.(2) on page 156]{kasparov1988} 
      in combination with the 
      translation \cite[Lemma 7.6.(iv)]{aspects}, but 
      the latter transformation employs 
      Kasparov's technical 
      theorem involving the axiom of choice and thus destroys the computability 
      via explicit formulas in our approach.  
       
       {\bf (b)} 
      However, notice that for {\em non-equivariant} $KK$-theory all $G$-actions are trivially very special and thus all 
      formulas for the $KK$-theory products of the first column
      of the above  table work, involving typically 
      computations in the Dirac dual-Dirac method. 
      One at first stabilizes the Kasparov cycles by adding on the infinite sum Hilbert modules $A^\infty$ to them in order to get 
      only  ordinary inverse corner embeddings  $e^{-1} : \calk  
      \otimes A  
      \rightarrow A$ within them. 
      The computability by explicit formulas 
      then remains.  
      (For instance, $\phi$ in \re{eq55} below 
      may be chosen to be an injective algebra homomorphism 
      induced by the Kasparov stabilization theorem.)  
   
        {\bf Conjecture.}
        {\bf (a)}  
        Similarly, it is believable that Lemma \ref{lemma21}  
        might be 
        provable for compact groups $G$ and $KK^G$ by an 
        analogous proof, which would yield computability of those Kasparov products  
       positively 
       affirmed in the very special $GK^G$ column 
        of the above table  
        of Section \ref{sec7}   
        within $KK^G$ 
        for {\em compact} groups $G$, 
        because $KK^G$ 
        would then automatically be {\em very special} 
        $KK^G$-theory.  

        {\bf (b)} 
      Even more so, it appears that in many $KK$-theory element 
       constructions used by 
      Kasparov \cite{kasparov1988, zbMATH06685586} and others,   
      the underlying Hilbert modules are constructed as shown in the table 
      of Subsection \ref{subsec211}, and this offers a way to lift the computability 
      of these examples to {\em $G$-equivariant} $KK$-theory 
      once results from 
      the note \cite{corner} are extended to Hilbert $C^*$-modules - as we would expect 
      to be possible -, because inverse corner embeddings 
      would then be describable 
      as $\phi \cdot f^{-1}$   
       for an algebra homomorphism $\phi$ and a {\em very special} 
 corner embedding $f$ 
 as  mentioned in Subsection \ref{subsec211}.  

Indeed, in that case 
a Kasparov cycle $z \in KK^G(A,B)$ involving only Hilbert modules constructed out of the elementary 
constructions shown in the table of Subsection \ref{subsec211} 
could be presented as 
a level-one element 
$t_+ f_2 f_1^{-1} \Delta_{t_- \oplus \id_A} \phi f^{-1} \in GK^G(A,B)$  
 in $GK^G$-theory 
  as visualized in this diagram:
\be{eq55}
\xymatrix{ 
B  \ar[r]^-f      &  \calk_B( \oplus_\N B) &  \calk_B(  \calf  \oplus B) 
\ar[l]_-\phi  
 	\ar[r]_-{\iota}   &  \call_B \big ( 
 	\calf \oplus B \big ) \square_{t_- \oplus \id_A}  A       \ar@<.5ex>[r] 
 	\ar@<-.5ex>[l]_-{ \Delta_{t_- \oplus \id_A} }   
&      A \ar[l]^-{t_\pm \oplus \id_A}    
}  
\en 

(For $\phi$ the identity this 
had already  been described in 
\cite[Definition 9.3]{aspects}.) 

Because $f$ is very special, 
this would make some products 
involving factors of 
such cycles computable,  if calculable 
according to the table 
of Section \ref{sec7}
meeting extra criteria like (very) speciality of $M_2$-spaces 
and/or finiteness of the matrix size of involved corner embeddings.

{\bf Outlook.} 
The reverse implication of 
the last 
considerations, technically already workable for algebras, is that example computations from $KK^G$-theory 
like those showing up in the Dirac 
dual-Dirac method  \cite{kasparov1988}
might be potentially transferred - at least as far as 
calculating the Kasparov 
product at the first place is concerned - to other categories of algebras than $C^*$-algebras using $GK^G$-theory. 
%
Related seems Cuntz' paper \cite{zbMATH05130839} about 
index theory 
on 
differentiable manifolds 
employing $kk$-theory 
for locally convex algebras, 
and {Corti{\~n}as and Tartaglia \cite{zbMATH06823303}.

\bibliographystyle{plain}
\bibliography{references}

\begin{thebibliography}{10}

\bibitem{zbMATH07828316}
Yavar Abdolmaleki and Dan Kucerovsky.
\newblock A short proof of an index theorem. {II}.
\newblock {\em J. Noncommut. Geom.}, 18(1):123--142, 2024.

\bibitem{zbMATH07155161}
Iakovos Androulidakis and Georges Skandalis.
\newblock A {Baum}-{Connes} conjecture for singular foliations.
\newblock {\em Ann. \(K\)-Theory}, 4(4):561--620, 2019.

\bibitem{zbMATH07308574}
Paolo Antonini, Sara Azzali, and Georges Skandalis.
\newblock The {Baum}-{Connes} conjecture localised at the unit element of a
  discrete group.
\newblock {\em Compos. Math.}, 156(12):2536--2559, 2020.

\bibitem{zbMATH07599758}
Francesca Arici and Jens Kaad.
\newblock Gysin sequences and {{\(SU(2)\)}}-symmetries of
  {{\(C^\ast\)}}-algebras.
\newblock {\em Trans. Lond. Math. Soc.}, 8(1):440--492, 2021.

\bibitem{zbMATH07127525}
Francesca Arici and Adam Rennie.
\newblock The {Cuntz}-{Pimsner} extension and mapping cone exact sequences.
\newblock {\em Math. Scand.}, 125(2):291--319, 2019.

\bibitem{zbMATH07358369}
Alexandre Baldare.
\newblock The index of {{\(G\)}}-transversally elliptic families. {I}.
\newblock {\em J. Noncommut. Geom.}, 14(3):1129--1169, 2020.

\bibitem{gk1}
B.~{Burgstaller}.
\newblock {The generators and relations picture of $KK$-theory}.
\newblock 2016.
\newblock arXiv:1602.03034v2.

\bibitem{gk}
B.~{Burgstaller}.
\newblock {A kind of $KK$-theory for rings}.
\newblock 2021.
\newblock arXiv:2107.01597.

\bibitem{gk2}
B.~{Burgstaller}.
\newblock {On the $K$-theory in splitexact algebraic $KK$-theory}.
\newblock 2024.
\newblock arXiv:2410.04150.

\bibitem{corner}
B.~{Burgstaller}.
\newblock {Corner embeddings into algebras of compact operators in $K$-theory}.
\newblock 2025.
\newblock arXiv:2501.11504.

\bibitem{aspects}
Bernhard Burgstaller.
\newblock Aspects of equivariant {{\(KK\)}}-theory in its generators and
  relations picture.
\newblock {\em Adv. Oper. Theory}, 10(2):30, 2025.
\newblock Id/No 32.

\bibitem{conneshigson2}
A.~{Connes} and N.~{Higson}.
\newblock {Almost homomorphisms and $KK$-theory}.
\newblock 1990.
\newblock Downloadable manuscript.

\bibitem{conneshigson}
A.~{Connes} and N.~{Higson}.
\newblock {D\'eformations, morphismes asymptotiques et \(K\)-th\'eorie
  bivariante}.
\newblock {\em {C. R. Acad. Sci., Paris, S\'er. I}}, 311(2):101--106, 1990.

\bibitem{cortinasthom}
G.~{Corti\~nas} and A.~{Thom}.
\newblock {Bivariant algebraic \(K\)-theory}.
\newblock {\em {J. Reine Angew. Math.}}, 610:71--123, 2007.

\bibitem{zbMATH06823303}
Guillermo Corti{\~n}as and Gisela Tartaglia.
\newblock Compact operators and algebraic {{\(K\)}}-theory for groups which act
  properly and isometrically on {Hilbert} space.
\newblock {\em J. Reine Angew. Math.}, 734:265--292, 2018.

\bibitem{cuntz}
J.~{Cuntz}.
\newblock {Bivariante \(K\)-Theorie f\"ur lokalconvexe Algebren und der
  Chern-Connes-Charakter}.
\newblock {\em {Doc. Math.}}, 2:139--182, 1997.

\bibitem{cuntzweyl}
J.~{Cuntz}.
\newblock {Bivariant \(K\)-theory and the Weyl algebra}.
\newblock {\em {\(K\)-Theory}}, 35(1-2):93--137, 2005.

\bibitem{cuntzthom}
J.~{Cuntz} and A.~{Thom}.
\newblock {Algebraic \(K\)-theory and locally convex algebras}.
\newblock {\em {Math. Ann.}}, 334(2):339--371, 2006.

\bibitem{zbMATH05130839}
Joachim Cuntz.
\newblock An algebraic description of boundary maps used in index theory.
\newblock In {\em Operator algebras. The Abel symposium 2004. Proceedings of
  the first Abel symposium, Oslo, Norway, September 3--5, 2004}, pages 61--86.
  Berlin: Springer, 2006.

\bibitem{zbMATH06295993}
Marius Dadarlat.
\newblock Group quasi-representations and almost flat bundles.
\newblock {\em J. Noncommut. Geom.}, 8(1):163--178, 2014.

\bibitem{zbMATH05530175}
Claire Debord and Jean-Marie Lescure.
\newblock {{\(K\)}}-duality for stratified pseudomanifolds.
\newblock {\em Geom. Topol.}, 13(1):49--86, 2009.

\bibitem{zbMATH05248319}
Siegfried Echterhoff, Heath Emerson, and Hyun~Jeong Kim.
\newblock {KK}-theoretic duality for proper twisted actions.
\newblock {\em Math. Ann.}, 340(4):839--873, 2008.

\bibitem{ellis}
E.~{Ellis}.
\newblock {Equivariant algebraic \(KK\)-theory and adjointness theorems}.
\newblock {\em {J. Algebra}}, 398:200--226, 2014.

\bibitem{zbMATH07366975}
Chi-Kwong Fok and Varghese Mathai.
\newblock The ring structure of twisted equivariant {{\(KK\)}}-theory for
  noncompact {Lie} groups.
\newblock {\em Commun. Math. Phys.}, 385(2):633--666, 2021.

\bibitem{garkusha1}
G.~{Garkusha}.
\newblock {Algebraic Kasparov \(K\)-theory. I}.
\newblock {\em {Doc. Math.}}, 19:1207--1269, 2014.

\bibitem{garkusha2}
G.~{Garkusha}.
\newblock {Algebraic Kasparov \(K\)-theory. II}.
\newblock {\em {Ann. \(K\)-Theory}}, 1(3):275--316, 2016.

\bibitem{grensing}
M.~{Grensing}.
\newblock {Universal cycles and homological invariants of locally convex
  algebras}.
\newblock {\em {J. Funct. Anal.}}, 263(8):2170--2204, 2012.

\bibitem{higson}
N.~{Higson}.
\newblock {A characterization of KK-theory}.
\newblock {\em {Pac. J. Math.}}, 126(2):253--276, 1987.

\bibitem{higson2}
N.~{Higson}.
\newblock {Categories of fractions and excision in KK-theory}.
\newblock {\em {J. Pure Appl. Algebra}}, 65(2):119--138, 1990.

\bibitem{zbMATH07498656}
Jens Kaad and Valerio Proietti.
\newblock Index theory on the {Mi{\v{s}}{\v{c}}enko} bundle.
\newblock {\em Kyoto J. Math.}, 62(1):103--131, 2022.

\bibitem{kasparov1981}
G.~G. {Kasparov}.
\newblock {The operator K-functor and extensions of C*-algebras}.
\newblock {\em {Math. USSR, Izv.}}, 16:513--572, 1981.

\bibitem{kasparov1988}
G.~G. Kasparov.
\newblock {Equivariant $KK$-theory and the Novikov conjecture}.
\newblock {\em Invent. Math.}, 91:147--201, 1988.

\bibitem{zbMATH06685586}
Gennadi Kasparov.
\newblock Elliptic and transversally elliptic index theory from the viewpoint
  of {{\(KK\)}}-theory.
\newblock {\em J. Noncommut. Geom.}, 10(4):1303--1378, 2016.

\bibitem{arXiv:2210.02332}
Gennadi Kasparov.
\newblock K-theory and index theory on manifolds with a proper {Lie} group
  action.
\newblock Preprint, {arXiv}:2210.02332 [math.{KT}] (2022), 2022.

\bibitem{zbMATH01997275}
Gennadi Kasparov and Georges Skandalis.
\newblock Groups acting properly on ``bolic'' spaces and the {Novikov}
  conjecture.
\newblock {\em Ann. Math. (2)}, 158(1):165--206, 2003.

\bibitem{zbMATH07227224}
Bram Mesland and Mehmet~Haluk {\c{S}}eng{\"u}n.
\newblock Hecke operators in {{\(KK\)}}-theory and the {{\(K\)}}-homology of
  {Bianchi} groups.
\newblock {\em J. Noncommut. Geom.}, 14(1):125--189, 2020.

\bibitem{zbMATH07089428}
Shintaro Nishikawa.
\newblock Direct splitting method for the {Baum}-{Connes} conjecture.
\newblock {\em J. Funct. Anal.}, 277(7):2259--2287, 2019.

\bibitem{phillips}
N.~C. Phillips.
\newblock {{\(K\)}}-theory for {Fr{\'e}chet} algebras.
\newblock {\em Int. J. Math.}, 2(1):77--129, 1991.

\bibitem{zbMATH07173644}
Koen van~den Dungen.
\newblock The index of generalised {Dirac}-{Schr{\"o}dinger} operators.
\newblock {\em J. Spectr. Theory}, 9(4):1459--1506, 2019.

\bibitem{weidner}
J.~{Weidner}.
\newblock {KK-groups for generalized operator algebras. I}.
\newblock {\em {\(K\)-Theory}}, 3(1):57--77, 1989.

\bibitem{weidner2}
J.~{Weidner}.
\newblock {KK-groups for generalized operator algebras. II}.
\newblock {\em {\(K\)-Theory}}, 3(1):79--98, 1989.

\bibitem{zbMATH06317160}
Zhizhang Xie and Guoliang Yu.
\newblock Positive scalar curvature, higher rho invariants and localization
  algebras.
\newblock {\em Adv. Math.}, 262:823--866, 2014.

\end{thebibliography}

\end{document}